\documentclass[twoside]{article}
\usepackage{amsmath,amssymb}
\def\letitre{Spatially periodic instantons: Nahm transform and moduli}   %
\title{\letitre}
\author{Benoit Charbonneau, %
Jacques Hurtubise}
\date{14 December 2017}

\newcommand{\addressBenoit}{Department of Pure Mathematics, 
				University of Waterloo, 
				200 University Avenue West, Waterloo, Ontario, N2L 3G1, Canada.}    
\newcommand{\emailBenoit}{benoit@alum.mit.edu}          
\newcommand{\addressJacques}{Department of Mathematics and Statistics,
             			McGill University,
           			805 Sherbrooke St. W, Montreal, Quebec, H3A 2K6, Canada.}
\newcommand{\emailJacques}{jacques.hurtubise@mcgill.ca}

\usepackage[margin=2.5cm,includehead]{geometry}
\usepackage{fancyhdr}
\fancypagestyle{plain}{%
   \fancyhead[LO]{}
   \fancyhead[CO]{}\fancyhead[RO]{}
}

\newenvironment{proof}{\topsep=\smallskipamount \partopsep=0pt  %
 \begin{trivlist} \itemindent=\parindent                        %
 \item[\hskip \labelsep\emph{Proof:}]}{\qed\end{trivlist}}      %
\let\qed=\relax                                                 %
\def\qed                                                        %
 {{\unskip\nobreak\hfil\penalty50                               %
   \quad\hbox{}\nobreak\hfil $\Box$                             %
   \parfillskip=0pt \finalhyphendemerits=0 \par}}               %

\def\@thmcountersep{-}                                          %
\newtheorem{theorem}{Theorem}[section]                          %
\newtheorem{proposition}[theorem]{Proposition}                  %
\newtheorem{prop}[theorem]{Proposition}
\newtheorem{corollary}[theorem]{Corollary}                      %
\newtheorem{lemma}[theorem]{Lemma}                              %
\usepackage{diagrams}  % Paul Taylor's package
        \diagramstyle[noPS]
          \newarrow{Iso}C---{>>}
	\newarrow{Corresponds}<--->

\DeclareMathAlphabet{\mathdj}{U}{msb}{m}{n}
\newcommand{\R}{\ensuremath{\mathdj {R}}}  % R, Real numbers
\newcommand{\C}{\ensuremath{\mathdj {C}}}  % C, complex numbers
\newcommand{\Z}{\ensuremath{\mathdj {Z}}}  % Z, integers
\newcommand{\bbr}{\R}			   % R, Jacques's macro
\newcommand{\uu}{{\mathfrak{u}}}   % u
\newcommand{\SU}{\mathrm{SU}}     % SU
\newcommand{\U}{\mathrm{U}}        % U

\newcommand{\PP}{\ensuremath{\mathdj {P}}} % P, the Projective space
\newcommand{\RT}{{\R\times S^1\times T}}         % RxT^3
\newcommand{\RTn}[2]{{\R^{#1}\times T^{#2}}}
\newcommand{\RnS}[1]{{\R^{#1}\times S^1}}

\newcommand{\ch}{\mathrm{ch}}           % Chern Character
\newcommand{\Ind}{\mathrm{Ind}}         % Index bundle
\newcommand{\End}{\mathrm{End}}         % Endomorphism bundle
\newcommand{\EE}{{\mathcal{E}}}         % E of holomorphic bundle
\newcommand{\FF}{\mathcal{F}}  		    % F, some sheaf 
\newcommand{\FFF}{\mathfrak{F}}
	\newcommand{\FFu}{\FFF}  % F, upper route
	\newcommand{\FFl}{{\FF}} % F, lower route
\newcommand{\II}{\mathcal{I}}  			% I, some sheaf
\newcommand{\KK}{\mathcal{K}}         % K, some sheaf
\newcommand{\KKK}{{\mathfrak{K}}}         
	\newcommand{\KKu}{{\KKK}}	% K, upper route
	\newcommand{\KKl}{{\KK}}	% K, lower route
\newcommand{\LL}{{\mathcal{L}}}         % L, some sheaf
\newcommand{\PPP}{{\mathcal{P}}}        % P, the Poincare bundle
\newcommand{\QQ}{{\mathcal{Q}}}         % Q, some sheaf
\newcommand{\OO}{{\mathcal{O}}}         % O, the sheaf of holomorphic functions

\newcommand{\diag}{\mathrm{diag}}        % diagonal matrix
\newcommand{\iso}{\cong}                 % iso
\newcommand{\rk}{\mathrm{rk}}            % rank
\newcommand{\MM}{{\mathcal{M}}}		 % M for moduli space
\newcommand{\Pic}{{\mathrm{Pic}}}

\newcommand{\SSS}{{\mathcal{S}}}         % S for spectral transform
\newcommand{\spinC}[1]{\,\,\,\makebox[0pt]{$#1$}\raisebox{1pt}{\makebox[0pt]{$\slash$}}\hspace{3pt}}  
\newcommand{\DD}{\spinC{\mathfrak{D}}}            %% DIRAC OPERATOR          
\newcommand{\db}{\bar\partial}                    % d bar
\newcommand{\del}{\bar{\partial}}                 % d bar again
\newcommand{\p}{\partial}                         % partial

\newcommand{\bump}{\chi}                  % bump function
\newcommand{\Eh}{\hat{E}}                 % E hat 
\newcommand{\floor}[1]{\lfloor#1\rfloor}  % floor fct
\newcommand{\dist}{\mathrm{dist}}         % dist
\newcommand{\id}{\mathbf{I}}              % Identity map
\newcommand{\Lip}{\mathrm{Lip}}           % Lip
\newcommand{\loc}{\mathrm{loc}}           % loc
\newcommand{\radius}{\mathrm{radius}}     % radius
\newcommand{\sll}{\mathrm{sl}}            % slope sl
\newcommand{\Spec}{\mathrm{Spec}}         % Spec
\newcommand{\ord}{\mathrm{ord}}         % ord

\newcommand\ca{\mathcal}
\newcommand{\citep}[2]{\cite[{#1}]{#2}}
\usepackage[usenames,dvipsnames]{xcolor}
\usepackage[pdffitwindow=false,colorlinks, citecolor=blue, linkcolor=blue, urlcolor=Maroon, filecolor=Maroon]{hyperref}
\usepackage{url}

\newcommand{\fold}[1]{\textcolor{blue}{\textbf{---Folded---}}}
\newcommand{\folded}[2]{\textcolor{blue}{\textbf{---Folded (#1)---}}}
\renewcommand{\fold}[1]{#1}
\renewcommand{\folded}[2]{#2}

\newcommand{\ok}{}

\begin{document}
\maketitle

%%%%%%%%%%%%%%%%%%%%%%%%%%%%%%%%%%%%%%%%%%%%%%%%%%%%%%%
%%%%%%                                           %%%%%%
%%%%%%          ABSTRACT                         %%%%%%
%%%%%%                                           %%%%%%
%%%%%%%%%%%%%%%%%%%%%%%%%%%%%%%%%%%%%%%%%%%%%%%%%%%%%%%
\begin{abstract}This paper establishes that the Nahm transform sending spatially periodic instantons (instantons on the product of the real line and a three-torus) to singular monopoles on the dual three-torus is indeed a bijection as suggested by the heuristic.  In the process, we show how the Nahm transform intertwines to a   Fourier--Mukai transform  via   Kobayashi--Hitchin correspondences.  We also  prove existence and non-existence results.
\end{abstract}
\footnotetext{The authors can be reached respectively at [\emailBenoit, \addressBenoit], [\emailJacques, \addressJacques]}

%%%%%%%%%%%%%%%%%%%%%%%%%%%%%%%%%%%%%%%%%%%%%%%%%%%%%%%
%%%%%%                                           %%%%%%
%%%%%%          SECTION                          %%%%%%
%%%%%%                                           %%%%%%
%%%%%%%%%%%%%%%%%%%%%%%%%%%%%%%%%%%%%%%%%%%%%%%%%%%%%%%
\section{Introduction} % (fold)
\label{sec:introduction}
This paper examines the intertwining, in an interesting geometric case, of three recurrent themes of the study of the anti-self-duality equations: the Nahm transform, the Kobayashi--Hitchin correspondence, and the Fourier--Mukai transform.

\textbf{The Nahm transform.} The Nahm transform is a heuristic which, in its basic form, relates
solutions to the self-duality or anti-self-duality equation  on a quotient $X$ of
$\R^4$ by a closed subgroup $\Lambda$ to   solutions on a quotient $X^*$ of 
${\R^4}^*$ by the dual subgroup $\Lambda^*:=\{f\in{\R^4}^*\mid
f(\Lambda)\subset\Z\}$.  Because some of these quotients are not four-dimensional, one has to consider reductions of self-duality: the Bogomolny, Hitchin, and Nahm equations come into play, for example.

The different possibilities for correspondence $X\leftrightarrow X^*$ are given, up to diffeomorphism, by transposing the following diagram:
\begin{equation*}\label{allquotients}\begin{matrix}
   &T^4 &{\RTn{}3} & \RTn22 & {{\RnS3}} &\R^4   &\\
   &{{T^3}} &{\RTn{}2} & \RTn21 & \R^3\\
   &T^2 &\RnS{} &  \R^2\\
  &{{S^1}} & \R    &         &       &X\\
 &{*}^{}&     &         &  X^*
  \end{matrix}\end{equation*}
 The metrics on the tori, while all flat, depend on the $\Lambda$ chosen. We note that the transform interchanges $\Lambda$ with its dual.

In a nutshell, the heuristic starts with a connection $\nabla$ satisfying the  ASD equation (or the appropriate dimensional reduction). One solves the Dirac equation in the background of a family of shifts of $\nabla$ by characters parameterized by $X^*$.  Glued together,  the spaces of solutions form a vector bundle $\Eh$ over $X^*$, embedded in a   trivial infinite rank bundle of $L^2$ sections of the appropriate spinor bundles. Projecting on $\Eh$ the trivial connection of the trivial $L^2$ bundle,  one obtains the transformed connection, and the necessary endomorphisms are produced by projecting the  multiplication by the coordinates of the non-compact directions of $X$. This description is the cartoon picture; the actual transforms involve fixing appropriate singularities, boundary conditions, etc.  A detailed description of the heuristic in our case is given in \cite[Section 3]{benoitpaper}.

The transform was introduced by Nahm in \cite{nahmEqn,Nahm} to explain the ADHM construction of instantons over $\R^4$ introduced in
\cite{ADHM} in physicists' language.   Nahm then extended the heuristic to monopoles \cite{Nahm}. The heuristic given by Nahm
was set on firmer mathematical ground by Corrigan--Goddard in
\cite{corrigan-goddard}, and Hitchin in \cite{hitchinMonopoles}.    

The transform provides a framework in which to think about the
classification of all the  solutions to the ASD
equation whose curvature has finite $L^2$-norm. This framework, well explained in the survey paper \cite{jardimsurvey} and in \cite[\S 5.2]{FM-Nahm-book},  guided
several authors in the understanding of 
moduli spaces of instantons (or their appropriate dimensional
reduction) on various quotients  $X$ of $\R^4$:
\cite{ADHM,biquardjardim,braam1989,benoitpaper,benoitjacques2,Durcan-Cherkis,Cherkis-Kapustin-Super-QCD,CherkisWard-MonopoleWalls,corrigan-goddard,hitchinMonopoles,hurtubiseMurray,hurtubiseMurraySpectral,jardim2002,nakajima,nye,Osborn-ADHM,schenk1988,Szabo-SMF}.  While the heuristic gives the general idea, carefully working out the details has proven to be necessary.  One notable example is the case of $\R^2\times T^2$ where to study the full moduli space of finite $L^2$ curvature instantons, one has to introduce wild singularities on the codomain of the transform, and in turn one has to get surjectivity one has to increase the domain to a larger set of filtered bundles.  Mochizuki's paper \cite{Mochizuki-doublyperiodic} recounting this story settled negatively a conjecture of Jardim (in \cite{jardim2002b}) on the asymptotic decay of doubly-periodic instantons, conjecture that had been standing for more than a decade.

The heuristic has been implemented for other families of spaces, first for the asymptotically locally Euclidean complete hyper-K\"ahler manifolds (the \emph{ALE gravitational instantons}) by
Kronheimer--Nakajima in \cite{kronheimer1990}.
Extending the Nahm transform to other gravitational instantons had to wait for the corresponding flat case to be solved first. Cherkis proposed a framework to tackle the \emph{ALF} case in   \cite{Cherkis-InstantonsGravitons}, informed by earlier successes on the Taub--NUT space in \cite{Cherkis-Taub-NUT,Cherkis-instantonsTN}.  This framework involves a mix of of Nahm and ADHM data on the representation of a bow.  At this stage, we do not yet know that the Nahm transform is bijective for ALF spaces (except for  the simplest  \cite{Cherkis-Hurtubise}) but the transform from instantons to bow representations is established on solid footing in \cite{Cherkis-Larrain-Stern-1}.

A general framework where  the moduli space of flat connections that is normally the target of the transform  is replaced by a moduli space of instanton connections has been proposed by \cite{Bartocci2004a,Bartocci2004} to give a transform between moduli spaces of instantons on various  hyper-K\"ahler manifolds, not all four dimensional. There is great current interest in studying instantons in higher dimension, and a Nahm transform in this context is understandably a desired tool given its success in lower dimension. Very little has been accomplished at this time, except perhaps for \cite{Nakamula-Sasaki-Takesue-ADHM}.

In another direction,  the Nahm transform gauge theoretic version of \cite{Bonsdorff-thesis,Bonsdorff-Crelles} for Higgs bundles has been implemented in \cite{Frejlich-Jardim}.

One notable hole in the literature is the case of $\R\times T^3$, addressed in this paper.   After  early work by van Baal (numerical approximations and remarks in \cite{vanBaal1996} and a computation in the case of charge 1 in \cite{vanBaal1999}), the singular monopole obtained by the Nahm transform in the case of structure group $\SU(2)$ was studied further by the first author in his PhD thesis \cite{benoitthesis,benoitpaper}.  Physical motivation for getting a deeper understanding of the Nahm transform in this case can be found in the recent paper \cite{Maxfield-Sethi-domain-walls} of Maxfield and Sethi.

\textbf{The Kobayashi--Hitchin correspondence.} If one is on a compact K\"ahler surface, part of the ASD equations implies the integrability of the $\bar\partial$-operator associated to the connection, and so gives a holomorphic vector bundle, which can be shown to be stable when the connection is irreducible \cite[Thm 2.3.2]{Lubke-Teleman}.
 The correspondence between irreducible ASD connections (mod gauge) and stable bundles turns out to be a bijection; this was first proven by Donaldson \cite{donaldson-surfaces}. The idea is that by solving a variational problem one can find a metric whose Chern connection has  ASD curvature. The notion of antiselfduality is the two complex dimensional version of a general property of a connection in arbitrary dimension, that of being Hermitian--Yang--Mills (HYM); there is then a correspondence, in arbitrary dimension between HYM connections and stable holomorphic bundles. This correspondence goes back to Narasimhan and Seshadri in dimension one, and is due to Uhlenbeck and Yau in arbitrary dimensions. If one adds in singularities, boundaries, symmetries,  Higgs fields, etc., the correspondence, appropriately modified, often still holds, and has been the subject of works by a wide variety of authors; let us mention Simpson \cite{Simpson-Hodge-structures}, Biquard \cite{Biquard-fibresparaboliquesstables} 
and Biquard--Jardim \cite{biquardjardim} in particular, as their results are most particularly relevant our present discussion. For a more thorough history of the Kobayashi--Hitchin correspondence, one can read the introduction of \cite{Lubke-Teleman}.

\textbf{The Fourier--Mukai transform.} For the Nahm transform, we twisted by a family of flat connections, and looked for solutions to the Dirac equation; the holomorphic analogue here is to twist by all line bundles and take a direct image onto the Jacobian. This transform, introduced by Mukai in the early eighties as a way of obtaining equivalences between derived categories on Abelian varieties \cite{FM}, has become a standard element of the algebro-geometric toolkit; see, for instance, the very good books of Bartocci,  Bruzzo, and  Hernandez Ruip\'erez, or of Huybrechts \cite{FM-Nahm-book,Huybrechts-FM}.It is an interesting question of what the Nahm transform does to the holomorphic data that classifies the ASD connections. In the four-torus case, one obtains a Fourier--Mukai transform.  Similar things occur in the case that interests us.

\textbf{The cases at hand.}
Let $T$ be a two-torus, equipped with a flat metric; it is in a natural way a genus one curve. We set 
$T^*$ to be the dual torus, with $t^*\in T^*$ corresponding to a line bundle $L_{t^*}$ on $T$.
We  study several classes of data, showing that they are equivalent, with the relations fitting into the following commuting diagram:
\begin{equation}\label{basic-diagram}
\begin{diagram} 
  (F,\nabla,\phi)\text{ on }S^1\times T^*& &&  \rCorresponds^{\text{Nahm} }&&&\ (E,\nabla) \text{ on }  \RT\\ 
\dCorresponds>{\text{Kob.--Hit.}} &&&&&& \dCorresponds<{\text{Kob.--Hit.}}\\ 
\begin{matrix}\text{Pair } (\FFu,\rho)\\ {\rm on}\ T^*\end{matrix}&\rCorresponds{\rm Hecke\quad }&
\begin{matrix}\text{Pair } (\FFl,\rho)\\ \text{ on } T^*\end{matrix}&\rCorresponds{\text{Spectral}}&\begin{matrix}(C, \KKl)\\ \text{ on }\PP^1 \times T^*\end{matrix}&\rCorresponds{\quad\text{Fourier--Mukai}}&\begin{matrix}\text{Bundle } \EE\\ \text{ on } \PP^1 \times T\end{matrix}
\end{diagram}\hskip-5mm\end{equation} 

Here (more complete definitions are given later):
\begin{itemize}
\item $(F,\nabla,\phi)$ is a $\U(k)$ monopole, defined on $ S^1\times T^*$, with Dirac type singularities of weight    $(1,0,\ldots,0)$ at $n$ points $(  \theta^+_i, t^{*,+}_i)\in  S^1\times T^*$, and of weight    $(-1,0,\ldots,0)$ at $n$ points $(\theta^-_i, t^{*,-}_i)\in  S^1\times T^*$. The points $t^{*,+}_i, t^{*,-}_i$ are supposed distinct.
\item $(E,\nabla)$ is a rank $n$ bundle with a finite energy charge $k$ anti-self-dual $\U(n)$ connection on the flat cylinder  $\RT$. At $+\infty$,   $(E,\nabla)$ is asymptotic to a fixed flat $\U(1)^n$ connection on $S^1 \times T $ which, on each $ \{\mu\}\times T$, corresponds to the sum of line bundles $\oplus_i L_{t^{*,+}_i}$, and on the $S^1$ factors, acts on the $L_{t^{*,+}_i}$ by $\frac{\partial}{\partial \mu} +  i\theta^+_i$. One has, at $-\infty$, the same, but with $(t^{*,-}_i, \theta^-_i)$.
\item $(\FFu ,\rho)$, and $(\FFl ,\rho)$ are pairs consisting of a rank $k$ holomorphic vector bundle over $T$ equipped with a meromorphic automorphism $\rho$, with simple zeroes at ${t^{*,+}_i}$, simple poles at ${t^{*,-}_i}$; the map $\rho$ is an isomorphism elsewhere. The pairs satisfy stability conditions involving the $\theta^\pm_i$.
\item $(C, \KKl )$ is a pair consisting of a holomorphic curve $C$ of degree $(n,k)$ in $ \PP^1 \times T^*$, and a sheaf  $\KKl$ supported on $C$. The curve $C$ intersects $\{0\} \times T^*$ in the points $(0, t^{*,+}_i)$, and 
$\{\infty\}\times T^*$ in $(\infty, t^{*,-}_i)$.
\item $\EE $ is a semi-stable holomorphic rank $n$ vector bundle of degree $0$ over $ \PP^1 \times T$, with $c_2(\EE ) = k$. Over $\{\infty\}$ in $\PP^1$, the bundle $\EE $ is the sum $\oplus_i L_{t^{*,+}_i}$, and over $\{0\}$, it is $\oplus_i L_{t^{*,-}_i}$
\item The vertical correspondences are the  Kobayashi--Hitchin correspondences, which associate to anti-self dual connections some holomorphic data that classify them,
\item The horizontal correspondences are Fourier-type transforms, the top one using the Dirac equation, and the bottom holomorphic data.
\end{itemize}

The sections, in what follows, in essence explore the various pieces of the diagram (\ref{basic-diagram}) in turn. Section \ref{sec:monopolesandpairs} begins with the left hand side with its vertical arrow, recalling the results of \cite{benoitjacques3}. Section \ref{sec:instantonsandholomorphicbundles} considers the right hand side, examining the Kobayashi--Hitchin correspondence for this case.  Section \ref{sec:FourierMukai} treats the bottom row, and discusses the Fourier--Mukai transforms. Section \ref{sec:NahmTransform} considers the top row, recalls results of Charbonneau \cite{benoitpaper} on the Nahm transform in this case, shows that the diagram commutes, and sums up the equivalences. Section \ref{sec:non_moduli} derives a few consequences on moduli, and discusses a few remaining questions.

% section introduction (end)

%%%%%%%%%%%%%%%%%%%%%%%%%%%%%%%%%%%%%%%%%%%%%%%%%%%%%%%
%%%%%%                                           %%%%%%
%%%%%%          SECTION                          %%%%%%
%%%%%%                                           %%%%%%
%%%%%%%%%%%%%%%%%%%%%%%%%%%%%%%%%%%%%%%%%%%%%%%%%%%%%%%
\section{The leftmost vertical arrow: monopoles and pairs $(\FFu ,\rho)$ \ok}\label{sec:monopolesandpairs}
In this section we consider the left hand side of the basic diagram (\ref{basic-diagram}), and describe the correspondence between monopoles and pairs $(\FFu ,\rho)$ which is the subject of  \cite{benoitjacques3}, specialized to the case that concerns us. We begin with a definition of a family of singular monopoles over $S^1\times T^*$. Let $t^* $ be a holomorphic coordinate on $T^*$. 
 We suppose chosen distinct points $t^{*,+}_i, t^{*,-}_i, i=1,\ldots, n$ of $T^*$.
 
 For simplicity, we suppose that $S^1$ is equipped with a metric giving it circumference $1$, and denote by $\theta\in \R$  the corresponding multi-valued coordinate. Let $\floor{\ }$ denotes the integer part. Also let $\widetilde\theta\in [0,1)$ parametrize  the circle $S^1$ isometrically, with $\widetilde\theta=\theta-\floor{\theta}$.  As constants associated to these angular coordinates reappear throughout the paper, we pause to give some of our conventions. 
 
 We have points in the circle represented by $\theta_i^+, \theta_i^-, i = 1,\ldots,n$, which for convenience we  suppose distinct, ordered, and not integers. One has  
\begin{equation}\sum_i\theta_i^+-\sum_i\theta_i^-=0;\end{equation}
 a priori this number could be any integer but we shall see that the normalisation to zero is natural. With our conventions, we set
\begin{align*} 
	0&<\theta_1^+<\theta_2^+<\cdots<\theta_n^+<\theta_n^+\leq\theta_1^++1<2,\\
0&<\theta_1^-<\theta_2^-<\cdots<\theta_n^-<\theta_n^-\leq\theta_1^-+1<2.\end{align*}
  We set 
  \begin{align*}
  	\widehat\theta^+_i&:= \theta_i^+ +1,   & \widehat\theta^-_i&:= \theta_i^-,\\
\widetilde\theta^\pm_i&:= \theta_i^\pm - \floor{\theta_i^\pm},&\widehat\Theta &:= \sum_i\floor{\theta_i^-}-\floor{\theta_i^+}.
  \end{align*}

Let $p_1^+,\ldots,p_n^+$ be the collection of points on $S^1\times T^*$ given by $p_i^+=( \widetilde \theta^+_i  , t^{*,+}_i)$; similarly, set $p_i^- =(\widetilde\theta_i^-   , t^{*,-}_i)\in  S^1\times T^*$.    We fix weights $w^+ = (w_{1}^+ ,\ldots,w_{ k}^+ ) =(1,0,\ldots,0)$, $w ^- = (w_{1}^- ,\ldots,w_{k}^-) = (-1,0,\ldots,0)$; these weights are to be thought of as cocharacters, mapping $\U(1)$ to the maximal torus.  Let $F$ be a $\U(k)$-bundle on $(S^1\times T^*)\setminus\{p_i^+, p_i^-\mid i=1,\ldots,n\}$ whose degree is $\widehat\Theta$ on $\widetilde\theta=0$, is $ 1$ on spheres around the $p_i^+$, and is  $ -1$ on spheres around the $p_i^-$.  As one moves the complex curves $\widetilde\theta=cst$ through the $p_i^\pm$, the degree of the restriction of $F$ thus changes by $\pm 1$.  Let $\nabla$ be a $\U(k)$-connection on $F$, and let $\phi$ be a section of the associated (adjoint) $\uu(k)$ bundle. Following \cite{benoitjacques3}, we  say that $(F,\nabla,\phi)$ is a singular $\U(k)$ monopole on $(S^1\times T^*)\setminus\{p_i^+, p_i^- \mid i=1,\ldots,n\}$ of weights $w_i^\pm$ at $p_i^\pm$ if
\begin{itemize}
\item $(F,\nabla,\phi)$ satisfy the Bogomolny equation $[\nabla_i,\nabla_j] = \sum_k \epsilon_{ijk}\nabla_k\phi$  (equivalently $*F_\nabla=d_\nabla \phi$), and if
\item we have Dirac type singularities in a neighbourhood   of the  $p_i^\pm$: in a neighbourhood of $p_i^\pm$, if $R$ is the geodesic distance to the singularity, one imposes that
 in a suitable gauge
\begin{align*}
\phi =& \frac {i}{2R}\diag(w_{i,1}^\pm,\ldots,w_{i,n}^\pm) + O(1)= \frac {i}{2R}\diag(\pm1,0,0,..,0) + O(1),
%\text{and }\nabla (R \phi) =& O(1)
\end{align*}
and $\nabla (R \phi) = O(1)$ as $R\to 0$.
\end{itemize}

We can define from such a monopole a pair  $(\FFu ,\rho)$   consisting of a rank $k$ holomorphic vector bundle over $T^*$, and $\rho$ a meromorphic automorphism of $\FFu$. For the first, we simply note that the connection defines a $\bar \partial$ operator $\nabla^{0,1}_{T^*}$ over the Riemann surface $T^*$ given by $ \theta= 0$, giving the bundle $F|_{\theta= 0}$ a holomorphic structure which we denote by $\FFu$. The Bogomolny equations imply that 
$[\nabla^{0,1}_{T^*}, \nabla_\theta - i\phi] = 0$, so that parallel transport defined by integrating $ (\nabla_\theta - i\phi)s= 0$ preserves the holomorphic structure. Transporting around the circle defines a return map $\rho \colon \FFu \to \FFu $.  It is holomorphic as long as one stays away from the $t^{*\pm}_i$. It is shown in \cite[Prop 2.5]{benoitjacques3} that in a neighbourhood of the  $t^{*\pm}_i$  one can write $\rho = a_i^\pm(t^*) \diag_j ((t^*- t^{*\pm}_i)^{w_{i,j}^\pm}) b_i^\pm(t^*)$ with $ a_i^\pm,  b_i^\pm$  holomorphic and invertible. The pair $(\FFu ,\rho)$ can be shown to satisfy a stability condition: for  a $\rho$-invariant subbundle $\FFu '$ of $\FFu$ of rank $k'$, degree $j_0'$, one lets the weights of the restriction $\rho'$ of $\rho$ to $\FFu'$ at $p_i^\pm$ be denoted by $w_{i,j}^{\pm,{}'}, j= 1, \ldots ,k'$, and one can define a degree 
\begin{equation}\delta_{\widetilde\theta^{\pm}}(\FFu' ,\rho') = \deg(\FFu') -\sum_{i,\pm} \widetilde\theta_i^\pm\sum_j (w_{i,j}^{\pm,{}'}),\label{eqn:degreedef}\end{equation}
and a slope
\begin{equation}\sll_{\widetilde\theta^\pm}(\FFu' ,\rho') = \frac{\delta_{\widetilde\theta^{\pm}}(\FFu' ,\rho')}{\rk(\FFu')}.\label{eqn:slopedef}\end{equation}
Note that $\sum_j (w_{i,j}^{\pm,{}'})$ is the order of $\det(\rho')$ at $t^{*,\pm}$; in our case it is $0$ or $+1$ at the points $t^{*,+}$, and $0$ or $-1$ at the points $t^{*,-}$.
  One then has the (semi-)stability condition by asking that for all invariant subbundles $\FFu'$,
\[\sll_{\widetilde\theta^\pm}(\FFu',\rho')\quad (\leq)< \quad \sll_{\widetilde\theta^\pm}(\FFu ,\rho).\]
Here, as $\rho$ has poles, invariance of a subsheaf $\FFu'$ is defined by asking that   any section of $\FFu'$ has image under $\rho$ lying in $\FFu'$ as long as the image lies in $\FFu$. An irreducible monopole gives a stable pair $(\FFu ,\rho)$; if the monopole is reducible, one gets a sum of stable $(\FFu ,\rho)$  with same slope. As the stability depends on $\widetilde\theta_i^\pm$, we refer to this notion of stability as $\widetilde\theta^\pm$-stability.

We have from \cite[Remark 3.8]{benoitjacques3} the proposition:
\begin{proposition} For bundles coming from monopoles and so defined on the complement of $p_i^\pm$ in $S^1\times T^*$, the quantity $  \sll_{\widetilde\theta^\pm}(\FFu ,\rho)\cdot \rk(\FFu)$ is the average in $\widetilde\theta$ of the degree of the restriction of $\FFu$ to the curves $\{\widetilde\theta\}\times T^*$ as one moves $\widetilde\theta$ around the circle. \end{proposition}

For the case of concern to us, we have
\begin{equation}  \sll_{\widetilde\theta^\pm}(\FFu ,\rho) =0. \end{equation}

The main theorem of \cite{benoitjacques3}, Theorem 4.1 (specialized to our case), states that the monopole to stable pair correspondence map is a bijection:
\begin{theorem}\label{thm:main}
The moduli space ${\MM}_{\widehat\Theta}^{ir}(S^1\times T^*,
p_1^\pm,\ldots,p_n^\pm,w_1^\pm \ldots,w_n^\pm)$ of $\U(k)$
irreducible   monopoles on
$S^1\times T^*$ with $F|_{\widetilde\theta=0}$ of degree
$\widehat\Theta$ and Dirac type singularities at $p_j^\pm$ of type $w_i^\pm$ maps bijectively to the
space ${\MM}_s\bigl(T^*,\widehat\Theta, \bigl((w_i^\pm,t_i^{*,\pm}),i=1, \ldots,n\bigr),\widetilde\theta^\pm\bigr)$ of
 $\widetilde\theta^\pm$-stable holomorphic pairs $(\FFu,\rho)$ with 
\begin{itemize}
\item $\FFu$ a holomorphic rank $k$ bundle of degree $\widehat\Theta$ on $T^*$,
\item $\rho$ a meromorphic section of $\mathrm{Aut}(\FFu)$ of the form $a_i^\pm(t^*)
\diag_j ((t^*- t^{*,\pm}_i)^{w_{i,j}^\pm})b_i^\pm(t^*)$ near $t_i^{*,\pm}$, with  $ a_i^\pm,  b_i^\pm$  holomorphic and
invertible; 
$\det(\rho)$ has divisor $\sum_{i } t_i^{*,+}-t_i^{*,-}$.
\end{itemize}
More generally, 
the reducible  monopoles  correspond bijectively to 
 $\widetilde\theta^\pm$-polystable, but unstable, pairs.
\end{theorem}
(The theorem is proven in \cite{benoitjacques3} for $\widetilde\theta_i^\pm$ distinct, but this restriction was simply for notational convenience, and removing it poses no problem; it is essential however that the $t^{*,+}_i, t^{*,-}_i$ be distinct; also, one can have different integer weights, and slopes that are not integers, though for the latter one must generalise the Bogomolny equation.)

%%%%%%%%%%%%%%%%%%%%%%%%%%%%%%%%%%%%%%%%%%%%%%%%%%%%%%%
%%%%%%                                           %%%%%%
%%%%%%          SECTION                          %%%%%%
%%%%%%                                           %%%%%%
%%%%%%%%%%%%%%%%%%%%%%%%%%%%%%%%%%%%%%%%%%%%%%%%%%%%%%%
\section{The rightmost vertical arrow: instantons and holomorphic bundles}\label{sec:instantonsandholomorphicbundles}
%%%%%%%%%%%%%%%%%%%%%%%%%%%%%%%%%%%%%%%%%%%%%%%%%%%%%%%
%%%%%%                                           %%%%%%
%%%%%%          SUBSECTION                       %%%%%%
%%%%%%                                           %%%%%%
%%%%%%%%%%%%%%%%%%%%%%%%%%%%%%%%%%%%%%%%%%%%%%%%%%%%%%%
\subsection{Holomorphic extension \ok} % (fold)
\label{sec:holomorphic_extension}
We have a finite energy instanton $(E, \nabla)$ on the cylinder $ \RT$, and we want to define a suitable bundle $\EE$ on $  \PP^1\times T$  corresponding to it. 

Let therefore $(E, \nabla)$ be a bundle with a unitary anti-self-dual connection on $ \RT$. We suppose that the energy ($L^2$ norm of the curvature) of the connection is finite, which gives us a charge $k$ for the second Chern class, and some good asymptotic behaviour.

Indeed, let $\phi,\psi$ be angular coordinates on $T$, $\mu$ be an angular coordinate on the $S^1$, and $s$ a standard coordinate on the $\R$. As this $S^1$ is dual to the $S^1$ of the previous section, its circumference is $2\pi$. Write $\nabla$ as $d + iA_sds +iA_\mu d\mu + iA_\psi d\psi + iA_\phi d\phi$. Fortunately, the asymptotics of $A$ has been well developed in the context of Floer theory, in various places (\cite{Donaldson2002, mrowkabook, Taubes1993}). Let us just deal with the end $s = +\infty$, the treatment of the other end being the same.
In a gauge with $A_s = 0$, the instanton is asymptotic  at infinity  to a  flat unitary connection $\partial_\mu + iA^\infty_\mu, \partial_\psi+ iA^\infty_\psi, \partial_\phi + iA^\infty_\phi$. We  suppose that we are in a generic situation: 
\begin{enumerate}
	\item 
that the connection on $T$ at $s=\infty$ corresponds to a sum 
\[L_{t_1^{*,+}}\oplus\cdots \oplus L_{t_n^{*,+}}\]
 of distinct line bundles (this situation corresponds in our monopole picture to the locations of the singularities $t_i^{*,+}$ being distinct and also gives the condition of regularity; see \cite{Donaldson2002}),
 and  
 \item that the connection on $S^1$ has monodromy with distinct eigenvalues. 
  \end{enumerate}
 In a gauge with $A_s= 0$, we have for some $\rho>0$ that
\begin{equation}\begin{aligned} \nabla_\mu =&\ \partial_\mu + iA^\infty_\mu + O(\exp(-\rho s)),\\
\nabla_\psi =&\ \partial_\psi + iA^\infty_\psi + O(\exp(-\rho s)),\\
\nabla_\phi =&\ \partial_\phi + iA^\infty_\phi + O(\exp(-\rho s)). \end{aligned}\label{asymptotics}\end{equation}
 The connection matrices at infinity $iA^\infty_\mu,   iA^\infty_\psi,  iA^\infty_\phi$ are constant, and commute. 
 Also, all derivatives of the connection matrices also have exponential decay, with the same $\rho$. By hypothesis, the eigenvalues of $  A^\infty_\mu$ are distinct, and we  suppose that $  A^\infty_\mu$ is diagonal, with real eigenvalues  $\theta^+_1, \dots, \theta^+_n$, with $\theta^+_1<\theta^+_2<\cdots<\theta^+_n\leq \theta^+_1+1 $.
  The other connection matrices are then also diagonal. Note that gauge transformations asymptotic at infinity to  $ \exp(i \mu \diag_i (n_i))$, which can be different at either end of the cylinder, shift the $\theta_i$ by integers. 

We suppose similar behaviour at $s=-\infty$, with $t_i^{*,-}, \theta_i^-$ instead of $t_i^{*,+}, \theta_i^+$. The   $\theta_i^+, \theta_j^-$ are supposed distinct.  We also  note that the determinant bundle is a $\U(1)$ instanton, hence flat; this fact gives the natural normalisation $\sum_i\theta_i^+=\sum_i\theta_i^-$.

We complexify our picture: $T$ is to be thought of as an elliptic curve, with complex coordinate $w$; the cylinder $\R\times S^1$ becomes complex, too, with a coordinate $z= \exp(-s-i\mu)$, so that $s= \infty$ is $z=0$. The complex cylinder $\R\times S^1$ has a two-point compactification to $\PP^1$. An instanton on $\RT$ defines a holomorphic structure on the bundle, and the aim is to extend it to $\PP^1\times T$. 

As is typical when dealing with moduli spaces of holomorphic bundles, one needs a notion of stability.  This notion is described immediately following the theorem.

%%%%%%%%%%%%%%%%% THM
\begin{theorem}\label{thm1} Let $(E,A)\to \RT$ be a rank $n$ instanton of charge $k$,
with asymptotic
behaviour given as above (that is by Equation (\ref{asymptotics}) and the following paragraph).  Then there is a holomorphic bundle $\EE\to \PP^1\times T$ such that
\[\EE|_{\RT}\iso (E,\del_A),\] 
with the extension to the compactification $\EE$  chosen so that
\begin{enumerate}
\item $c_1(\EE)=0$ and $c_2(\EE)=k$,
\item       $\EE|_{\{+\infty\}\times T}\iso\oplus_{i=1}^n L_{t_i^{*,+}}$ and $\EE|_{\{-\infty\}\times T}\iso  \oplus_{i=1}^n L_{t_i^{*,-}}$,  (thus, on the elliptic curves $T_0,T_\infty$ over the points $0$, $\infty$ in $\PP^1\times T$ and so for the generic  $T_z$, the  bundle $\EE$ is a sum of line bundles of degree zero), and
\item $\EE$ is $\theta^\pm$-stable.
\end{enumerate}
\end{theorem}

The stability criterion here is one of parabolic type: one considers only subbundles $\EE'$ that are sums of line bundles of degree zero for $z=0,\infty$; this condition forces $\EE'$ to be a sum $\oplus_{i\in I^+} L_{t_i^{*,+}}$ over $z=0$, and $\oplus_{i\in I^-} L_{t_i^{*,-}}$ over $z= \infty$ (that is, they are built from a  subset of the line bundles involved in the direct sum decomposition of $\EE$ at the corresponding infinity), and also $c_1(\EE') = j'[\omega_T]$. We define the ${\theta^\pm}$-degree of $\EE'$ to be 
\begin{equation}\label{eqn:deltathetapm}\delta_{\theta^\pm}(\EE') = j' -\sum_{i\in I^+} \theta_i^+ + \sum_{i\in I^-} \theta_i^-,\end{equation}
with a corresponding slope $\sll_{\theta^\pm}(\EE') = \delta_{\theta^\pm}(\EE')/\rk(\EE')$. The definition of stability is the usual one: we ask that $\sll_{\theta^\pm}(\EE') < \sll_{\theta^\pm}(\EE) $ for all subbundles of the type we have restricted to. The slope $\sll_{\theta^\pm}(\EE) $ turns out to be zero.

%%%%%%%%%%%%%%%%% PROOF

%%%%%%%%%%%%%%%%% PROOF
Let us first  change coordinates. We set
\begin{equation}  z = \exp(-s - i \mu), \bar z = \exp(-s+i\mu)
\end{equation}
and choose an appropriate constant complex combination $w$ of $\phi$ and $\psi$ such that $w$ is a complex coordinate compatible with the conformal structure of $T^2$.
One then has $z \bar z = \exp(-2s), \bar z/z = \exp(2i\mu)$, and so
\begin{equation} d\mu = \frac{1}{2i} (-\frac{dz}{z} +\frac {d\bar z}{\bar z}).
\end{equation}
Thus, for the $(0,1)$ component of the connection $\nabla^{0,1} = (\frac{\p}{\p \bar z} + A_{\bar z})d\bar z+(\frac{\p}{\p \bar w} + A_{\bar w})d\bar w$, we have
\begin{align}A_{\bar z} &=  \frac{1}{2} \frac{A^\infty_\mu}{\bar z} + O(|z|^{\rho-1}),\label{untransformed}\\
A_{\bar w} &=  A^\infty_{\bar w} + O(|z|^{\rho}).\nonumber\end{align}
The terms with an $\infty$ superscript are constants.

Note that the notation is chosen here with a different convention in real or complex coordinates.  Indeed, we write $\nabla_{\bar z}=\frac{\p}{\p \bar z} + A_{\bar z} $ while we write $\nabla_{\mu}=\frac{\p}{\p \mu} + iA_{\mu}$ (notice the absence/presence of $i$).

We now want to show that this semiconnection is (complex) gauge equivalent locally in a neighbourhood of $z=0, w=0$ to 
\begin{align}A_{\bar z} &=  \frac{1}{2} \frac{A^\infty_\mu}{\bar z} ,\label{normalised} \\
A_{\bar w} &=   0.\nonumber \end{align}

We follow the ideas of Biquard and Jardim \cite{biquardjardim} (and ultimately from \cite[Section 9]{Biquard-caslogarithmique}), and consider   first   the planes $w = $\ constant. One wants to find a gauge transformation
$ g (z)= 1+u(z,\bar z)$, with $u$ smooth away from $z= 0$  such that $u/|z|$ is bounded. Writing $A_{\bar z}$ in Equation (\ref{untransformed}) as
\begin{equation}
A_{\bar z} =   \frac{1}{2} \frac{A^\infty_\mu}{\bar z} + a,	
\end{equation}
and applying $g$ to it, 
 we want
 \begin{equation}\label{eqn:for u complicated}
(1+u)\Bigl(\frac{1}{2} \frac{A^\infty_\mu}{\bar z} + a \Bigr)- \frac{\partial u}{\partial \bar z} = \Bigl(\frac{1}{2} \frac{A^\infty_\mu}{\bar z}\Bigr)(1+u), 	
 \end{equation}
 or equivalently $\nabla_{\bar z}=\p_{\bar z}+\frac12\frac{A_\mu^\infty}{\bar z}$ in the new gauge.
Let
\[\bar \partial_{A^\infty_\mu}(u) := \frac{\partial u}{\partial \bar z} + [ \frac{1}{2} \frac{A^\infty_\mu}{\bar z},u].\]
We note, as do \cite[p.~354]{biquardjardim}, that if $r^2(z) = z\bar z$, then for a matrix valued function $u$,  
\[\bar \partial_{A^\infty_\mu}(u) = r^{-A^\infty_\mu}\Bigl(\bar \partial\bigl(r^{A^\infty_\mu}\cdot u\cdot r^{-A^\infty_\mu}\bigr)\Bigr)r^{A^\infty_\mu}.\]
Equation \eqref{eqn:for u complicated} can be written simply as
\begin{equation}
	\label{eqn:elliptic eqn for u}
	\bar \partial_{A^\infty_\mu}(u) = ua +a. 
\end{equation}

A similar conjugation then holds for the Cauchy operators.  The Cauchy operators provide  inverses, at least for $C^1$ functions, but even here in our singular cases, as we now establish. Let $\bump\colon \C\to\R$ be a non-negative smooth function which is $1$ on $B(0,\frac R2)$ and $0$ outside $B(0, \frac 78R)$.   We consider the operators
\begin{align*}
K(f)(\zeta) &:=  \frac{1}{2\pi i}\int_\C \frac{ f }{z-\zeta}dz\wedge d\bar z,\\
\hat K &:= \bump K\circ \bump,\\
C(f) (\zeta) &=  K(f)(\zeta) -  K(f)(0),\\
\hat C &:= \bump C\circ \bump.
\end{align*}
These operator acts on scalar functions or on matrix valued functions $f$ by acting on each matrix entry $f_{ij}$ separately. Let $LT(f)$ denotes the lower triangular part of the matrix $f$.  We set 
\begin{align}
	\widetilde K(f)(\zeta) &:=  \hat K(f)(\zeta) -\hat K( LT(f))(0), \text{ and }\notag\\
	\widetilde K_{A^\infty_\mu}(f) &:= r^{-A^\infty_\mu} \widetilde K({r^{A^\infty_\mu}fr^{-A^\infty_\mu}})r^{ A^\infty_\mu}.\label{eqn:Ktilde}
\end{align}

On the smaller set $B(0,\frac R2)$, we have $\bar \partial_{A^\infty_\mu}\circ \widetilde K_{A^\infty_\mu}(f) = f$. We want 
 to solve 
\begin{equation}\label{eqn:def u implicit}
	u = \widetilde K_{A^\infty_\mu} (ua +a).
\end{equation}
While it could be prudent to carry $R$ along in our definitions (for instance, $\chi_R$, $\hat K_R$), we avoid doing so and the useful value of $R$ is determined by Lemma \ref{lemma:contraction} below.

The strategy to solve this equation is to introduce the operator
\[T(u):=\widetilde K_{A^\infty_\mu} (ua +a),\]
find an appropriate weighted Sobolev space domain for $T$ where it is a contraction, and show that the fixed point is smooth using elliptic regularity.   

For $X\subset \C$ (and $X=\C$ when omitted), consider the following norms
\begin{align*}
	\|f\|_{W^{\infty}_\alpha(X)}&:=\sup_{z\in X}\left||z|^{-\alpha}f(z)\right|,\\
	\|f\|_{W^{1,\infty}_\alpha(X)}&:= \|f\|_{W^\infty_\alpha(X)}+\|df\|_{W^\infty_{\alpha-1}(X)}.
\end{align*}
Note that the $df$ in the second definition can be a weak derivative.
Consider the function spaces
\begin{align*}
	W^{\infty}_{\alpha}(X)  =W^{0,\infty}_{\alpha}(X) &:=\{f\in L^1_\loc\mid \|f\|_{W^{\infty}_{\alpha}}<\infty\},\\
	W^{1,\infty}_{\alpha}(X) &:=\{f\in L^1_\loc\mid df\in L^1_\loc\text{ and }\|f\|_{W^{1,\infty}_{\alpha}}<\infty\},\\
	W^{k,\infty}_{c,\alpha}(X)&:=\{f\in W^{k,\infty}_{\alpha}(X)\mid \exists R\text{ s.t. }f(x)=0\text{ for }|x|>R\},\\
	W^{k,\infty}_{R,\alpha}(X)&:=\{f\in W^{k,\infty}_{\alpha}(X)\mid  f(x)=0\text{ for }|x|>R\},\\
	\widetilde W^{k,\infty}_{R,\alpha}&:=\text{closure of }C^1(\C\setminus\{0\})\text{ in }W^{k,\infty}_{R,\alpha}(\C).
\end{align*}

\textbf{Remarks:}
\begin{enumerate}
	\item Because of the introduction of weights, the $k$ in the standard notation $L^p_k$ could be interpreted both as a weight or as the number of weak derivative considered.  This issue suggests that we be deliberate in our choice of notation and weights are indices while number of derivatives are exponents.
	\item Elements of $W^{1,\infty}_\alpha$ are locally in $W^{1,\infty}$ hence  locally Lipschitz continuous, hence in $L^p_\loc$ (not just in $L^1_\loc$).  They therefore can be integrated against the Cauchy kernel.   
	\item The inequality $\alpha\geq \beta$ implies $W_{R,\alpha}^{k,\infty}\subset W_{R,\beta}^{k,\infty}$ and the operator norm of the inclusion is $R^{\alpha-\beta}$.
\item Multiplication by $|z|^\beta$ is an isomorphism $W^{k,\infty}_{\alpha}\to W^{k,\infty}_{\alpha+\beta}$ (and an isometry if $k=0$).
\item One should note that $W_0^\infty=L^\infty$ but $W_0^{1,\infty}\neq W^{1,\infty}$.
\end{enumerate}

\begin{lemma} %%%%%%%%%%%%%%%%%%%%%%%%%%%%%%%%%%%%%%%%%%%%%%%%%% LEMMA
Let $f\in W^{1,\infty}_{c,\alpha}$.  If $\alpha>-2$ then $K(f)$ converges.
\label{lemma:converge}
\end{lemma} %%%%%%%%%%%%%%%%%%%%%%%%%%%%%%%%%%%%%%%%%%%%%%%%%% END LEMMA

\begin{proof} %%%%%%%%%%%%%%%%%%%%%%%%%%%%%%%%%%%%%%%%%%%%%%%%%% PROOF
	\fold{
We use the fact that, as sets, $W^{1,\infty}=\Lip\subset C^0$.  Set $a\in\C\setminus\{0\}$. Let $U_j,V_j$ be open balls centered at $0$ if $j=0$ and $a$ if $j=1$ and such that $0\in V_0\subset U_0$ and $a\in V_1\subset U_1$, and $U_0\cap U_1=\emptyset$.  Let $U_2= \C\setminus \overline {V_0\cup V_1}$ and $V_2= \C\setminus \overline {U_0\cup U_1}$.  
 Suppose $\psi_1,\psi_2,\psi_3$ is a partition of unity subordinate to the covering $U_0,U_1,U_2$ with $\psi_j|_{V_j}=1$ for $j=0,1$. Let
$u_j(z):=\frac1{2\pi i}\int_{\C}\frac{(\psi_j f)(\lambda)}{\lambda-z}d\lambda\wedge d\bar\lambda$.  

For $z\in V_1$, the function $\zeta\mapsto\frac{\psi_2(\zeta)f(\zeta)}{\zeta-z}$ is in $C^0_c(\C)$. Thus $u_2(z)$ converges.

For $z\in V_1$,
\begin{align*}
	u_1(z)&=\frac1{2\pi i}\int_\C \frac{(\psi_1f)(\zeta)}{\zeta-z}d\zeta\wedge d\bar\zeta
% \\	&=\frac1{2\pi i}\int_\C \frac{(\psi_1f)(z+\zeta)}{\zeta}d\zeta\wedge d\bar\zeta\\ &
	=\frac1{\pi}\int_{0}^\infty\int_{0}^{2\pi} (\psi_1 f)(z+re^{i\theta})e^{-i\theta}d\theta dr.
\end{align*}
As $(\psi_1 f)$ is $C^0$ and compactly supported, the integral converges, hence $u_1\in C^0(V_1)$, as desired.

It remains to analyze $u_0$ and to prove that $u_0$ exists.  Let $U_0^\epsilon=U_0\setminus B(0,\epsilon)$ and let $u_0^\epsilon(z)=\frac1{2\pi i}\int_{U_0^\epsilon}\frac{(\psi_0f)(\lambda)}{\lambda-z}d\lambda\wedge d\bar \lambda$. For any $z\in V_1$, the integrand is in $C^0(U_0^\epsilon)$, so the integral exists.  We have, by definition, that  $u_0(z) = \lim_{\epsilon\to 0}u_0^\epsilon(z)$. 

Let $A_{\epsilon_1,\epsilon_2}=\overline{B(0,\epsilon_2)}\setminus B(0,\epsilon_1)$ when $\epsilon_1<\epsilon_2$. For $z\in V_1$, the norm of the denominator is bounded below by $\dist(V_1,U_0)>0$.  It is thus of no concern. 
 Since $f(\zeta)=O(|\zeta|^\alpha)$ then for $z\in V_1$ and some constant $k$,
\begin{align*}
|u_0^{\epsilon_1}(z)-u_0^{\epsilon_2}(z)|&\leq \frac1{2\pi i }\int_{A_{\epsilon_1,\epsilon_2}}\left|\frac{(\psi_0 f)(\zeta)}{\zeta-z}\right| d\zeta\wedge d\bar\zeta
%\\ &\leq \frac{2k}{\dist(V_1,U_0)} \int_{\epsilon_1}^{\epsilon_2} r^{\alpha+1}dr\\  &
\leq\frac{2k}{\dist(V_1,U_0)}\frac{|\epsilon_1^{\alpha+2}-\epsilon^{\alpha+2}_2|}{\alpha+2}.	
\end{align*}
Thus $u_0^{\epsilon}(z)$ is a Cauchy sequence as long as $\alpha>-2$ and therefore $u_0(z)$ exists.} 
\end{proof} %%%%%%%%%%%%%%%%%%%%%%%%%%%%%%%%%%%%%%%%%%%%%%%%%% END PROOF

%%%%%%%%%%%%% LEMMA
 \begin{lemma} \label{lemma:control}
% Assume that $R\leq 1$ and that $f=O_k(|z|^\alpha)$.  There exist a constant $k'$ independent of $f$ but dependent of $\alpha$ such that
%
% 1) if $\alpha>-2$ but $\alpha\neq -1$ then $K(f)=O_{kk'}(|z|^{\min(\alpha+1,0)})$,
%
% 2) if $\alpha>-1$ but $\alpha\neq 0$ then $C(f)=O_{kk'}(|z|^{\min(\alpha+1,1)})$.
%
We have that

1) if $\alpha>-2$ but $\alpha\neq -1$ then $K\colon W^\infty_{R,\alpha}\to W^\infty_{\min(\alpha+1,0)}$ is continuous,

2) if $\alpha>-1$ but $\alpha\neq 0$ then $C\colon W^\infty_{R,\alpha}\to W^\infty_{\min(\alpha+1,1)}$ is continuous.
\\
Moreover, their operator norms are bounded by a constant depending only on $\alpha$ when $R\leq 1$.
  \end{lemma}
%%%%%%%%%%% END LEMMA

\begin{proof}\folded{our old proof}{Let $f\in W^{\infty}_{R,\alpha}$ and let $k:=\|f\|_{ W^{\infty}_{R,\alpha}}$. Consider first the function $C(f)(\zeta)$. For simplicity, consider the large disk $D' = B(0,3R)$.  We know that $f|_{D'\setminus B(0,R)}=0$. We divide  $D'$ into 
%a disk $D_{|\zeta|/2, 0}$ centred at the origin, a bigger disk $D_{2|\zeta|, \zeta}$ centred at $\zeta$, from which we remove the first disk, and the complement $D'-D_{2|\zeta|, \zeta}$. We
into few parts and 
 estimate four integrals in turn, whose sum bounds $C(f)(\zeta)$:
\begin{align*}
\left|\frac{1}{2\pi i}\int_{B(0,|\zeta|/2)} \frac{f(z)}{z-\zeta}dz\wedge d\bar z \right|&\leq \frac{k}{2\pi }\int_{B(0,|\zeta|/2)} \frac{|z|^\alpha}{ |\zeta|/2} (idz\wedge d\bar z) =\frac{k}{2^\alpha(\alpha+2)}|\zeta|^{\alpha+1} 
\quad &(\forall\alpha>-2),\\
%%%%%%
\left|\frac{1}{2\pi i}\int_{B(\zeta,2|\zeta|)\setminus B(0,|\zeta|/2)} \frac{f(z)}{z-\zeta}dz\wedge d\bar z \right|&\leq \frac{k\max(3^\alpha,\frac1{2^\alpha})|\zeta|^\alpha}{2\pi }\int_{B(\zeta,2|\zeta|)} \frac{idz\wedge d\bar z}{|z-\zeta|} =  4 k\max(3^\alpha,2^{-\alpha_2})|\zeta|^{\alpha+1},
\\
%%%%%%
\left|\frac{1}{2\pi i}\int_{B(\zeta,2|\zeta|)} \frac{f(z)}{z}dz\wedge d\bar z \right|&\leq \frac{k}{2\pi }\int_{B(0,3|\zeta|)}|z|^{\alpha-1}(idz\wedge d\bar z)=\frac{2\cdot 3^{\alpha+1}}{\alpha+1}k |\zeta|^{\alpha+1}\quad &(\forall\alpha>-1),
\end{align*}
and
\begin{align*}\left|\frac{1}{2\pi i}\int_{D' \setminus B(0,2|\zeta|)} f(z) \bigg(\frac{1}{z-\zeta}-\frac{1}{z} \bigg)dz\wedge d\bar z\right|&\leq \frac{k|\zeta|}{2\pi }\int_{D'\setminus B(\zeta,2|\zeta|)} |z|^{\alpha-2}\biggl(\frac{|z|}{|z-\zeta|}\biggr)(idz\wedge d\bar z) \quad&\\
&\leq \frac{3k|\zeta|}{4\pi}\int_{D'\setminus B(0,|\zeta|)}|z| ^{\alpha-2} (idz\wedge d\bar z) \\
&=	       \frac{3k|\zeta|}{\alpha}\bigl((3R)^\alpha-|\zeta|^\alpha\bigr),\quad  &(\forall \alpha\neq 0),
\\
&\leq \begin{cases}\frac{3^{\alpha+1} R^{\alpha}k}{|\alpha|}|\zeta|,& \text{ if }\alpha>0,\\
  \frac{3k}{|\alpha|}|\zeta|^{\alpha+1},& \text{ if }\alpha<0.
\end{cases}
\end{align*}

For the function $K(f)(\zeta)$, we have the first two integrals above, as well as a bound
\begin{align*}
	\left|\frac{1}{2\pi i}\int_{D'\setminus B(\zeta,2|\zeta|)} \frac{f(z)}{z-\zeta}dz\wedge d\bar z\right|& \leq \frac k{2\pi} \int_{D'\setminus B(\zeta,2|\zeta|)} |z|^{\alpha-1}\biggl(\frac{|z|}{|z-\zeta|}\biggr)(idz\wedge d\bar z)\\
	&\leq \frac{3k\max(R^{\alpha+1},1)}{|\alpha+1|}|\zeta|^{\min(\alpha+1,0)}\quad&(\forall\alpha\neq -1).
\end{align*}
The proof is now complete.
}%end folded region
\end{proof}

\begin{lemma}\label{lemma:C1}
	Let $f\in C^1(\C\setminus\{0\})\cap W_{c,\alpha}^\infty$ for some $\alpha>-2$.    Then  $K(f)\in C^1(\C\setminus\{0\})$ and $\frac{\p K(f)}{\p \bar z}=f$.   
\end{lemma}

\begin{proof}
This result is the stated conclusion of H\"ormander's result \cite[Thm 1.2.2]{Hormander}\end{proof}

\begin{lemma} %%%%%%%%%%%%%%%%%%%%%%%%%%%%%%%%%%%%%%%%%%%%%%%%%% LEMMA
For all $R>0$, we have that

1) if $\alpha>-2$ but $\alpha\not\in \{-1,0\}$ then $\hat K\colon \widetilde W^{1,\infty}_{R,\alpha}\to \widetilde W^{1,\infty}_{R,\min(\alpha+1,0)}$ is continuous,

2) if $\alpha>-1$ but $\alpha\neq 0$, then $\hat C \colon \widetilde W^{1,\infty}_{R,\alpha}\to  \widetilde W^{1,\infty}_{R,\min(\alpha+1,1)}$ is continuous.\\
Moreover, their operator norms are bounded by a constant depending only on $\alpha$ when $R\leq 1$.
\end{lemma} %%%%%%%%%%%%%%%%%%%%%%%%%%%%%%%%%%%%%%%%%%%%%%%%%% END LEMMA

\begin{proof}That the target spaces are set correctly is guaranteed by Lemma \ref{lemma:C1}, we just need to see continuity.
Suppose now that $f\in C^1(\C\setminus\{0\})\cap W^{1,\infty}_{R,\alpha}$ and that $k=\|f\|_{W^{1,\infty}_{\alpha}}$.
We use the same notation as in the proof of Lemma \ref{lemma:C1}.    

The control on the $W^{\infty}_{\min(\alpha+1,0)}$-norm of $C(f)$ and $\hat K(f)$ comes from Lemma \ref{lemma:control}.
%\folded{bounding the derivative for Kf}{
We now tackle the derivative at the point $a$. We assume that $|\xi|=1$.  We take 
\begin{align*}
	V_0&=B(0,\frac{|a|}4),&U_0&= B(0,\frac{|a|}3),&
	V_1&=B(a,\frac{|a|}4),&U_1&= B(a,\frac{|a|}3).
\end{align*} 
Let $\tilde\psi$ be a bump function supported on $B(0,\frac13)$ and equal to $1$ on $B(0,\frac14)$.  Suppose that $C=\sup |d\tilde\psi|$.  Then let $\psi_0(z)=\tilde\psi(\frac z{|a|})$ and $\psi_1(z)=\tilde\psi(\frac{z-a}{|a|})$.  Then $|(D_\xi\psi_0)(z)|=\frac1{|a|} |(D_\xi\tilde\psi)(\frac{z}{|a|})|\leq \frac{C}{|a|}$ and $|(D_\xi\psi_1)(z)|=\frac1{|a|} |(D_\xi\tilde\psi)(\frac{z-a}{|a|})|\leq \frac{C}{|a|}$, thus $|(D_\xi\psi_2)(z)|\leq \frac{2C}{|a|}$ as $\psi_2=1-\psi_0-\psi_1$.  We thus get that
\[|d(\psi_if)(z)|\leq |\psi_i(z)df(z)+f(z)d\psi_i(z)|\leq (1+2C\frac{|z|}{|a|})k|z|^{\alpha-1}.\]

We have
\begin{align*}
	|D_\xi u_0(a)|&=\lim_{h\to 0}\left|\frac{\xi}{2\pi i}\int_{U_0}\frac{(\psi_0f)(\zeta)}{(\zeta-a-h\xi)(\zeta-a)} d\zeta d\bar\zeta\right|\\
	&\leq \lim_{h\to 0} \frac{2k}{(\alpha+2)3^{\alpha+2}}\frac{\radius(U_0)^{\alpha+2}}{\dist(V_1,U_0)^2}\leq \frac{18}{(\alpha+2)3^{\alpha+2}}k|a|^\alpha.
\end{align*}

We can have $D_\xi(\psi_1f)(a+re^{i\theta})\neq 0$ when $a+re^{i\theta}\in U_1= B(a,\frac{|a|}3)\subset B(0,2|a|)$, so only when $\frac{|a|}2\leq |a+re^{i\theta}|\leq 2|a|$.  Hence we have
\begin{align*}
	|D_\xi u_1(a)|&\leq\frac1\pi \int_0^\infty \int_0^{2\pi} |D_\xi(\psi_1f)(a+re^{i\theta})e^{-i\theta}| d\theta dr\\
%	&\leq (1+2C)k 2^{|\alpha-1|}|a|^{\alpha-1}\int_{0}^{\frac{|a|}3}dr \\
	&\leq \frac{(1+2C) 2^{|\alpha-1|}}{3}k|a|^\alpha.
\end{align*}

For $\lambda\in U_2$, we have $|a-\lambda|\geq \frac{|a|}{4}$. 
We have
\begin{align*}
	u_2(z)  %&=\frac1{2\pi i}\int_{\C} \frac{(\psi_2f)(\zeta)}{\zeta-z} d\zeta d\bar \zeta\\
	&= \frac1{2\pi i}\int_{\C} \frac{(\psi_2f)(z+\zeta)}{\zeta} d\zeta d\bar \zeta,
\end{align*}
thus
\[	(D_\xi u_2)(a)
	= \frac1{2\pi i}\int_{-a+U_2} \frac{D_\xi(\psi_2f)(a+\zeta)}{\zeta} d\zeta d\bar \zeta.\]
First note that $a+\zeta\in U_2$ implies that $|a+\zeta|\leq 5 |\zeta|$.
Since moreover $U_2\subset B(0,|a|/4)^c$, we have
\begin{align*}
	|(D_\xi u_2)(a)|
	%&= \left|\frac1{2\pi i}\int_{-a+U_2} \frac{D_\xi(\psi_2f)(a+\zeta)}{\zeta} d\zeta d\bar \zeta\right|\\
	&\leq \frac{10k(1+2C)}{2\pi} \int_{-a+U_2}|a+\zeta|^{\alpha-2}dxdy\\
		&\leq \frac{10k(1+2C)}{2\pi}\int_{B(0,R)\setminus B(0,\frac{|a|}4)}|\zeta|^{\alpha-2}dxdy\\
%	&\leq 10k(1+2C) \int_{|a|/4}^R r^{\alpha-1}dr\\
%	&=\frac{10k(1+2C)}{\alpha} (R^\alpha-|a|^p/4^\alpha)\\
	&\leq \begin{cases} \frac{10(1+2C)}{|\alpha|4^\alpha}k|a|^\alpha,&\text{ if }\alpha\in (-2,0),\\   \frac{10(1+2C)R^\alpha}{\alpha}k,&\text{ if }\alpha>0.\end{cases}
\end{align*}

%}%end folded command

Using this information and Lemma \ref{lemma:control}, we find for $\alpha\in(-2,\infty)\setminus\{-1,0\}$ a constant $c_{R,\alpha}$ such that
\[\|K(f)\|_{W^\infty_{\min(0,\alpha)}}\leq  c_{R,\alpha}\|f\|_{W^{\infty}_{\alpha}}\quad \text{ and }\quad   \|d(K(f))\|_{W^\infty_{\min(0,\alpha)}}\leq  c_{R,\alpha}\|f\|_{W^{1,\infty}_{\alpha}}.\]
Moreover, $c_{R,\alpha}\leq c_{1,\alpha}$ if $R\leq 1$.
Overall, we get that, if $f\in W^{1,\infty}_\alpha$  then $d(K(f))\in W^\infty_{\min(0,\alpha)}$. 
To prove the lemma, we need $d( K(f))\in W^\infty_{\min(0,\alpha+1)-1}=W^\infty_{\min(-1,\alpha)}$.  
By introducing the cut off function $\bump$ in the definition of $\hat K$, we can land in the right target space.   Let $c'_{R,\alpha}$ be the operator norm of the inclusion $W^\infty_{\min(0,\alpha)}\subset W^{\infty}_{\min(-1,\alpha)}$. We get
\begin{align*}
\|d(\hat K(f))\|_{W^\infty_{\min(0,\alpha+1)-1}}&=\|d(\hat K(f))\|_{W^\infty_{\min(-1,\alpha)}}\\
&\leq c'_{R,\alpha}\|d(\hat K(f))\|_{W^\infty_{\min(0,\alpha)}} \leq c'_{R,\alpha}c_{R,\alpha}\|f\|_{W^{1,\infty}_{\alpha}},	
\end{align*}
hence $\|\hat{K}(f)\|_{W^{1,\infty}_{\min(0,\alpha)}}\leq c'_{R,\alpha}c_{R,\alpha}\|f\|_{W^{1,\infty}_{\alpha}}$.

Note that $c'_{R,\alpha}=R$ if $\alpha>0$, $c'_{R,\alpha}=R^{\alpha+1}$ if $\alpha\in (-1,0)$, and $c'_{R,\alpha}=1$ if $\alpha\in(-2,-1)$. So certainly $c'_{R,\alpha}\leq 1$ for $R\leq 1$.

The treatment of $C$ is easier.  If $\alpha>-1$ but $\alpha\neq 0$,
\begin{align*}
	\|d(C(f))\|_{W^{\infty}_{\min(1,\alpha+1)-1}}&=	\|d(K(f))\|_{W^{\infty}_{\min(0,\alpha)}}
	\leq c_{R,\alpha}\|f\|_{W_\alpha^{1,\infty}},
\end{align*}
hence $\|C(f)\|_{W^{1,\infty}_{\min(1,\alpha)}}\leq c_{R,\alpha}\|f\|_{W^{1,\infty}_{\alpha}}$.
\end{proof}

For conceptual convenience, let's drop the superscript $\pm$ on the $\theta_i^\pm$.  Since for the moment we deal with the behaviour at $s=\infty$, let $\theta_i=\theta_i^+$.
Let $\mathrm{gaps}:=\{|\theta_i-\theta_j|\mid \theta_i\neq \theta_j\}$.
Let $\theta^{min}_{gap}:=\min \mathrm{gaps}$ and $\theta^{max}_{gap}:=\max \mathrm{gaps}$.  For any $\alpha$, let
\[\hat\alpha=\min(\alpha+1,\theta_{gap}^{min},1-\theta_{gap}^{max}).\]

\begin{corollary}\label{cor:controlf}
	% Suppose $\alpha\in(-1,1)$, that $\alpha\not\in -\mathrm{gaps}$ and that $\alpha+1\not\in \mathrm{gaps}$. For a matrix valued function $f$ such that $|f|\leq k|z|^\alpha$ we have
	% \[|\tilde K_{A^\infty_\mu}(f)|\leq kk'|z|^{\hat\alpha}.\]
Suppose $\alpha\in \Bigl((-1,0)\cup (0,\infty)\Bigr)\setminus \Bigl(\pm\mathrm{gaps}\cup (\mathrm{gaps}-1)\Bigr)$. Then for matrix valued functions, $\widetilde  K_{A^\infty_\mu}\colon \widetilde W^{1,\infty}_\alpha\to \widetilde W^{1,\infty}_{\hat\alpha}$ is continuous.
\end{corollary}

\begin{proof}
We check on the various components. Note that 
\begin{equation*}
	\bigl(\widetilde{K}_{A^\infty_\mu}(f)\bigr)_{ij}
	=\begin{cases}  |z|^{\theta_j-\theta_i}\hat C(|z|^{\theta_i-\theta_j}f_{ij}),& \text{ if }j\leq i,\\
					|z|^{\theta_j-\theta_i}\hat K(|z|^{\theta_i-\theta_j}f_{ij}),& \text{ if }i<j.
	\end{cases}
\end{equation*}

Suppose first that $0\leq \theta_j\leq \theta_i\leq 1$.  Then $\alpha+\theta_i-\theta_j>-1$.  We use here that $\alpha\not\in -\mathrm{gaps}$ and that $\alpha\neq 0$ when $\theta_i=\theta_j$.  Then the commutative diagram
\[\begin{diagram}
W_{\alpha}^{1,\infty}&&\rTo^{f_{ij}\mapsto (\tilde K_{A^\infty_\mu}(f))_{ij}}&&W^{1,\infty}_{\hat\alpha}\\
\dTo_{|z|^{\theta_i-\theta_j}}&&&&\uTo\\
W_{\alpha+\theta_i-\theta_j}^{1,\infty}&\rTo^{\hat C}&W_{\min(\alpha+\theta_i-\theta_j+1,1)}^{1,\infty}
&\rTo^{|z|^{\theta_j-\theta_i}}&W^{1,\infty}_{\min(\alpha+1,1+\theta_j-\theta_i)}
\end{diagram}\]
exhibits $f_{ij}\mapsto (\widetilde  K_{A^\infty_\mu}(f))_{ij}$ as a composition of continuous maps.

Suppose now that $0\leq \theta_i< \theta_j\leq 1$.  Then $\alpha+\theta_i-\theta_j>-2$. We use here both conditions that $\alpha\not\in \mathrm{gaps}$ and that $\alpha+1\not\in \mathrm{gaps}$. 
Then the commutative diagram
\[\begin{diagram}
W_{\alpha}^{1,\infty}&&\rTo^{f_{ij}\mapsto (\tilde K_{A^\infty_\mu}(f))_{ij}}&&W^{1,\infty}_{\hat\alpha}\\
\dTo_{|z|^{\theta_i-\theta_j}}&&&&\uTo\\
W_{\alpha+\theta_i-\theta_j}^{1,\infty}&\rTo^{\hat K}&W_{\min(\alpha+\theta_i-\theta_j+1,0)}^{1,\infty}
&\rTo^{|z|^{\theta_j-\theta_i}}&W^{1,\infty}_{\min(\alpha+1,\theta_j-\theta_i)}
\end{diagram}\]
exhibits $f_{ij}\mapsto (\widetilde  K_{A^\infty_\mu}(f))_{ij}$ as a composition of continuous maps.
The proof is now complete.\end{proof}
 
This corollary allows us to see  the need to use $C$ on the lower triangular part, in view of the conjugation involved in the definition of $\tilde K$ in Equation \eqref{eqn:Ktilde}.  If one had used $K$ throughout instead of $C$, when $\theta_i>\theta_j$, the function $|z|^{\theta_j-\theta_i}K(|z|^{\theta_i-\theta_j}f_{ij})$ would have growth rate $\min(\alpha+1, \theta_j-\theta_i)<0$, so would not be  bounded.

In view of solving Equation \eqref{eqn:def u implicit}, let $Z(u)=\widetilde{K}_{A^\infty_\mu} (ua +a)$.

\begin{lemma} \label{lemma:contraction}
	Suppose   $\alpha$ to be as in Corollary \ref{cor:controlf}, and suppose that $a\in W^{1,\infty}_{\alpha}$.  Then $Z\colon W^{1,\infty}_{R,\hat \alpha}\to W^{1,\infty}_{R,\hat \alpha}$ is   continuous.  Moreover, for a suitable $R<1$, this map is a contraction.
\end{lemma}

\begin{proof} %%%%%%%%%%%%%%%%%%%%%%%%%%%%%%%%%%%%%%%%%%%%%%%%%% PROOF
The map
\begin{align*}
	R_a\colon W^{1,\infty}_{R,\hat\alpha}&\to W^{1,\infty}_{R,\alpha+\hat\alpha}\\
	u&\mapsto ua
\end{align*}
is continuous with operator norm $\|R_a\|=\|a\|_{W^{1,\infty}_{\alpha}}$. Since $\hat\alpha>0$, the inclusion $\iota\colon W^{1,\infty}_{R,\alpha+\hat\alpha}\to W^{1,\infty}_{R,\alpha}$ is continuous, with operator norm $\|\iota\|=R^{\hat\alpha}$. 
Hence 
\begin{align*}
	W^{1,\infty}_{R,\hat\alpha}&\to W^{1,\infty}_{R,\alpha}\\
	u&\mapsto ua+a
\end{align*}
is continuous and therefore its composition $Z$ with the continuous map $\widetilde  K_{A_\mu^\infty}\colon W^{1,\infty}_{R,\alpha}\to W^{1,\infty}_{R,\hat\alpha}$ is continuous as desired.
 
For $u_1,u_2\in W^{1,\infty}_{R,\hat\alpha}$,  we have
\begin{align*}
	\|Z(u_1)-Z(u_2)\|_{W^{1,\infty}_{\hat\alpha}}&=\|\tilde K_{A^{\infty}_\mu}\bigl((u_1-u_2)a\bigr)\|_{W^{1,\infty}_{\hat\alpha}}\\
	&\leq \|\tilde K_{A^{\infty}_\mu}\| \|\bigl((u_1-u_2)a\bigr)\|_{W^{1,\infty}_{\alpha}}
	\leq \|\tilde K_{A^{\infty}_\mu}\| \|\iota\| \|R_a\| \|u_1-u_2\|_{W^{1,\infty}_{\hat\alpha}}\\
	&=C\|u_1-u_2\|_{W^{1,\infty}_{\hat\alpha}}
\end{align*}
for $C=\|\widetilde  K_{A^{\infty}_\mu}\| R^{\hat\alpha}\|a\|_{W^{1,\infty}_\alpha}$. By choosing $R$ small enough, we can ensure that $C<1$.
\end{proof} %%%%%%%%%%%%%%%%%%%%%%%%%%%%%%%%%%%%%%%%%%%%%%%%%% END PROOF	

\begin{corollary} %%%%%%%%%%%%%%%%%%%%%%%%%%%%%%%%%%%%%%%%%%%%%%%%%% COROLLARY
	\label{cor:fixed point of T}
Suppose   $\alpha$ to be as in Corollary \ref{cor:controlf}, and suppose that $a\in W^{1,\infty}_{\alpha}$. For suitable $R<1$, guaranteed by Lemma \ref{lemma:contraction}, there exists $u\in \widetilde W^{1,\infty}_{\hat\alpha}$  such that $\widetilde  K_{A^{\infty}}(ua+a)=u$.
\end{corollary} %%%%%%%%%%%%%%%%%%%%%%%%%%%%%%%%%%%%%%%%%%%%%%%%%% END COROLLARY

\begin{proposition} %%%%%%%%%%%%%%%%%%%%%%%%%%%%%%%%%%%%%%%%%%%%%%%%%% PROPOSITION
There exists a smooth bounded function $u$ such that $\bar\partial_{A^\infty_\mu}u=ua+a$ on $B(0,\frac R2)$.
\end{proposition} %%%%%%%%%%%%%%%%%%%%%%%%%%%%%%%%%%%%%%%%%%%%%%%%%% END PROPOSITION

\begin{proof} %%%%%%%%%%%%%%%%%%%%%%%%%%%%%%%%%%%%%%%%%%%%%%%%%% PROOF
The candidate is of course the fixed point $u$ of $Z$ guaranteed by Corollary \ref{cor:fixed point of T}.  
We know that on the smaller domain $B(0,\frac R2)$, the function $u$ is a $\widetilde W^{1,\infty}_{\hat\alpha}$ solution to Equation \eqref{eqn:elliptic eqn for u}.

We now observe $u$ in a compact ball $B\subset B(0,\frac R2)$ away from the singularity.  We have $u\in W^{1,\infty}(B)\subset W^{1,2}(B)$.  Since $u$ solves the elliptic equation \eqref{eqn:elliptic eqn for u}, we can bootstrap and obtain that $u$ is smooth on $B$.
\end{proof} %%%%%%%%%%%%%%%%%%%%%%%%%%%%%%%%%%%%%%%%%%%%%%%%%% END PROOF

Thus, we have the desired  gauge transformation on each disk $w$ = constant; but then,  it is straightforward to make this solution into a parametrised version (parametrised by $w$) and find the desired gauge given in Equation \eqref{normalised}.  

From this gauge, which is asymptotic to a unitary gauge, and in which the $\bar\partial$-operator is in a standard form $\nabla_{\bar z}=\p_{\bar z}+\frac12\frac{A^\infty_\mu}{\bar z}$,  one can pass to a holomorphic gauge where $\nabla_{\bar z}=\p_{\bar z}$ by a gauge transformation
\begin{equation}\label{unitary-to-holomorphic}g = \diag(r^{\theta_1 },r^{\theta_2},\ldots, r^{\theta_n}).\end{equation}
   We use this $g$ as a clutching function to extend the bundle over $z=0$. 

We can now check that the bundle we have produced has the required properties stipulated in Theorem \ref{thm1}. A first remark is that this (complex) gauge transformation preserves the connection in the $w$-direction, and so the bundle $\EE$ one obtains over $z=0$ is the one given by the operator $\bar \partial_w +  A^\infty_{\bar w}$; by our hypothesis, this restricted bundle is a sum of distinct line bundles $L_{t_i^{*,+}}$, and so $\EE$ does indeed have the structure we want over infinity. Secondly, the growth rates of unitary sections in a holomorphic basis as one approaches $z=0$ define  a flag $\EE_1\subset \EE_2\subset \cdots\subset \EE_n$ in $\EE$ over $z=0$. The flatness at infinity of our original connection (the matrices $ A^\infty_\mu, A^\infty_\psi, A^\infty_\phi$ commute), referring to the ordering of the  ${t_i^{*,+}}$,
 tells us that this flag coincides with the filtration
$L_{t_1^{*,+}}\subset L_{t_1^{*,+}}\oplus L_{t_2^{*,+}}\subset \cdots $, so that the flag is in essence an ordering of the line bundles $L_{t_i^{*,+}}$.

Now repeat all of this game at the other end of the cylinder, at  $s= -\infty$; one has opposite signs on the decay rates.   
%Relabel the $\theta_i$ associated to $s = +\infty$ as $\theta_i^+$, and those associated to $s = -\infty$ as $\theta_i^-$.

 We note that the determinant of $\EE$ is a $\U(1)$ instanton, decaying at infinity, and so is flat. From the boundary conditions at infinity, one already knows that the first Chern class on the tori is trivial. One can choose a basis along the cylinder so that the connection on the determinant bundle  is given by $\nabla_r = \partial_r, \nabla_\theta =\partial_\theta +  i\sum_j \theta_j^+ = \partial_\theta +  i\sum_j \theta_j^-$, so that   $\sum_i\theta_i^+ =\sum_i\theta_i^-$. Referring to the construction of the bundle $\EE$, one then has that $\det(\EE)$ is trivial. We note that one can have other choices of trivialisations at the ends, essentially shifting some of the $\theta_i^\pm$ by integers. This shift  corresponds to taking a Hecke transform of the bundle along $z=0$, or $z=\infty$. Independently of choices, however, one has $\sll_{\theta^\pm}(\EE) = 0$.
 
 There remains the question of stability.  One wants to investigate the  integral of the trace of curvature of subbundles. Indeed, as for example in \cite[Eqn.~(45)]{biquardjardim}, 
  this quantity  is the parabolic degree,  the topological degree being the sum of the curvature integral (the parabolic  degree) and a parabolic correction due in essence to the singular nature of the transition function (\ref{unitary-to-holomorphic}). Thus to show stability in our case, we want the curvature integral for a subsheaf on our cylinder to be negative. This negativity follows as in L\"ubke and Teleman \cite{Lubke-Teleman};
   the presence of an ASD connection implies that the curvature of a subbundle must have negative or zero integral. If  the degree of the bundle restricted to the elliptic curves is negative, then the total curvature is negative, as our cylinder is of infinite length, and so this case is automatic; on  the other hand, on the elliptic curves near the ends and so generically, the only possibility for a subbundle of positive or zero degree of $E$, which is a sum of line bundles of degree zero, is for the subbundle to be a sum of a subset of these line bundles, and so of zero degree when restricted to the elliptic curves; the curvature integral then in essence measures the (parabolic) degree in the $\PP^1$ direction.

%%%%%%%%%%%%%%%%% END PROOF

% section holomorphic_extension (end)

%%%%%%%%%%%%%%%%%%%%%%%%%%%%%%%%%%%%%%%%%%%%%%%%%%%%%%%
%%%%%%                                           %%%%%%
%%%%%%          SUBSECTION                       %%%%%%
%%%%%%                                           %%%%%%
%%%%%%%%%%%%%%%%%%%%%%%%%%%%%%%%%%%%%%%%%%%%%%%%%%%%%%%
\subsection{Stable holomorphic bundle to  instanton \ok}
Let us now examine the inverse problem, that of starting with a holomorphic bundle $\EE$ on $\PP^1\times T$, with a decomposition $\EE = \oplus L_{t_i^{*,\pm}}$ at $z= 0,\infty$, and producing an instanton on $\RT$. We suppose that the bundle $\EE$ satisfies the conclusion of Theorem \ref{thm1}: it has Chern classes $c_1(\EE) =0, c_2(\EE) = k$, it is $\theta^\pm$-stable, and is equipped with the decomposition at $z=0,\infty$ given as above by the $ L_{t_i^{*,\pm}}$.  

\begin{theorem} Let $\EE$ as above. There is a unique unitary bundle with anti-self-dual  connection $(E,A)$ on $\RT$ with curvature of finite $L^2$-norm and  trivializations near plus and minus infinity such that 
\begin{itemize}
\item $\EE|_{\RT}\iso (E,\del_A)$,
\item The connection $A$ has flat limits at plus or minus infinity :
\begin{align}  \nabla_\mu =&\ \partial_\mu + iA^{\pm \infty}_\mu  ,\\
\nabla_\psi =&\ \partial_\psi + iA^{\pm \infty}_\psi  ,\nonumber\\
\nabla_\phi =&\ \partial_\phi + iA^{\pm \infty}_\phi  ,\nonumber \end{align}
where the $A^{\pm \infty}$ are constant.
\item The relative Chern classes of $E$ are $c_1(E)= 0, c_2(E) = k$.
\end{itemize}
\end{theorem}
  
 \begin{proof} Our proof of this theorem follows a well established procedure, in essence following Donaldson \cite{donaldson-surfaces,Donaldson-boundary}, and Simpson \cite{Simpson-Hodge-structures}. Others who have dealt with the question of existence for non compact varieties include Guo \cite{guo,guo2}, Kronheimer--Mrowka \cite{Mrowka-Kronheimer-1,Kronheimer-Mrowka-embedded-Donaldson}, 
	    Biquard \cite{Biquard-fibresparaboliquesstables}, Owens \cite{owens}, Li--Narasimhan \cite{Li-Narasimhan-HE-Parabolic,Li-Narasimhan-Note-on-HE-Parabolic}.
 
 We want a connection whose $(0,1)$-part  coincides with that of the  bundle $\EE$; as we are building a Chern connection, it is then sufficient to give a metric.  To get this metric, the idea is to start with a suitable initial condition, and apply a heat flow; the infinite time value gives us the desired connection. We are in a different situation from the heat flows in \cite{Simpson-Hodge-structures},  essentially because our manifold  has infinite volume. Nevertheless, many of the same techniques apply, and our proof  follows Simpson's closely.

 {\it An initial condition. } For the bundle $E$, we  restrict the  bundle $\EE$ to  the cylinder $\RT$.
 Let us now consider trivializations near infinity. We  focus on the neighbourhood of $s= +\infty$, the case $s= -\infty$ being similar. For $s$ near infinity, the bundle $\EE$ is a sum   $\oplus_i L_{t_i^{*,+}(s,\mu)}$, given in suitable trivializations by   flat connections with constant connection matrices
 \begin{align}  
\nabla_\psi =&\ \partial_\psi + iA _\psi(s,\mu)  ,\nonumber\\
\nabla_\phi =&\ \partial_\phi + iA _\phi(s,\mu)  ,\nonumber \end{align}
 with constant coefficients along the elliptic curves.
On the complement of $t= +\infty$, we modify the trivialization by the inverse of the gauge transformation considered in Equation (\ref{unitary-to-holomorphic})
\begin{equation}g^{-1} = \diag(r^{-\theta_1^+ },r^{-\theta_2^+},\ldots, r^{-\theta_n^+}) = \diag(e^{-s\theta_1^+ },e^{-s\theta_2^+},\ldots, e^{-s\theta_n^+}).\end{equation}
 This modified trivialization is the unitary trivialization that we use for our initial condition. The corresponding Chern connection is
 \begin{align}  \nabla_\mu =&\ \partial_\mu + iA^{\pm \infty}_\mu  ,\nonumber\\
\nabla_\psi =&\ \partial_\psi + iA _\psi(s,\mu)  ,\nonumber\\
\nabla_\phi =&\ \partial_\phi + iA _\phi(s,\mu)  ,\nonumber\\
   \nabla_s =&\ \partial_s.\nonumber \end{align}
 
 We note that the holomorphic types and so the connection matrices $ iA _\psi(s,\mu), iA _\phi(s,\mu)$ decay at an exponential rate  in $s$ to their limits at infinity. Do the same in a neighbourhood of $s= -\infty$, and   choose  a metric  on a compact region $[-R,R]\times S^1\times T^2$, and patch the three metrics together using a partition of unity. The metric obtained is be our initial  value $H^0$ of the  metric which we use for the heat flow. Implicit in our choices is a homotopy class of trivialisations at infinity, and this fixes the relative Chern classes of the bundle.

  {\it Solving the heat equation on finite spatial intervals $[-R,R]$, for all positive times $t$. } Given a Hermitian metric $H$, let $F_H$ be the curvature of the Chern connection, and $\Lambda  F_H^\perp$ be its K\"ahler component. Following Simpson \cite{Simpson-Hodge-structures}, and Donaldson \cite{Donaldson-boundary} one solves the heat flow for the metric
 \begin{equation}
 	\label{eqn:heatflow}
	 H^{-1} \frac{dH}{dt} = - \sqrt{-1}\Lambda F_H^\perp.
 \end{equation}
  The proof of Simpson (\cite[Section 6]{Simpson-Hodge-structures}) goes through verbatim. One starts with a compact spatial interval $s\in [-R,R]$; Simpson shows that Equation \eqref{eqn:heatflow} can be solved for all time $t$ with either Dirichlet or Neumann boundary conditions, obtaining a flow $H_R(t)$ for $t\in[0,\infty)$. Simpson shows that $\sup(|\Lambda F_H^\perp|)$ decays in $t$; Donaldson shows in addition that it decays exponentially. 
There are   two maximum principles for the theory of instanton heat flow which are of great use. The first is on the functional $\sigma(H_1, H_2)$, a pseudo-distance between metrics; the second is on the quantity $|\Lambda F|$. Both are subharmonic on solutions, and so  the maximum is reached either at the boundary of the manifold, or in the initial conditions. 

  {\it Solving the heat equation on the full space $\RT$.}
  Again, following Simpson, Proposition 6.6 verbatim, one can solve the heat equation on the full space $\RT$, as a limit of the $H_R$; the solution is bounded on finite time intervals. (For his lemma 6.7, one can use the function $s^2$.) Let us call the solution $H$.
  
  {\it Convergence and stability}. One then must ask what happens as the time $t$ goes to infinity to our solution $H$. Let us concentrate on a large interval $s\in [-R,R]$. Again following \cite[Sections 5 and 7]{Simpson-Hodge-structures}, one can define the Donaldson functional $M(H^0, H(t))$; it decreases along the flow, and indeed satisfies
  \[\frac{d}{dt} M(H^0, H(t)) = -\int_{[-R,R]\times S^1\times T} |\Lambda F_{H(t)}^\perp|^2_{H(t)}.\]
  If $M(H^0, H(t))$ is unbounded below, one has, as in \cite[Section 5]{Simpson-Hodge-structures}, writing $H(t)$ as $H^0\exp (h(t))$, a subsequence $t_i$ such that $h(t_i)$ converges, after rescaling to unit norm, to a limit $u$ whose (constant) eigenspaces allow one to define a projection onto a holomorphic subbundle $E'$ of positive degree. The degree here is the one given by the curvature integral, that is, the parabolic degree. This projector is invariant in $R$. One now has one of three possibilities:
  
  1) The subbundle has negative degree when restricted to the tori $\{(s,\mu)\}\times T$. Hence the norm of the  curvature on each $\{(s,\mu)\}\times T$ is bounded below uniformly, and so the limiting connection has $L^p$ norm bounded below by a constant times $R$. As all our constructions involve limits with uniform bounds in $L^p$, this situation cannot occur.
  2) The subbundle has zero degree when restricted to the tori $\{(s,\mu)\}\times T$.  For $s$ large, the subbundle is basically a uniform choice $\oplus_{i\in I^+} L_{t_i^{*,+}(s,\mu)}$  for a subset $I^+$ of $\{1,\ldots,n\}$, and so the subbundle extends to infinity in a uniform way.
  By stability, the parabolic degree of the bundle would then have to be zero or negative, a contradiction
  3) The subbundle has positive degree when restricted to the tori $\{(s,\mu)\}\times T$. This case is impossible, as the bundle would have to map in a non-zero way to $\oplus_i L_{t_i^{*,+}(s,\mu)}$. 
  
  One then has, as in \cite[Section 7]{Simpson-Hodge-structures},  that the Donaldson functional is bounded below on ${[-R,R]\times S^1\times T}$, and so there is a subsequence of $t_i$ with $|\Lambda F_{H(t_i)}^\perp|^2_{H(t_i)}$ tending to zero. Elliptic theory then gives a smooth limit, solving the Hermite--Einstein equation. 
   \end{proof}

%%%%%%%%%%%%%%%%%%%%%%%%%%%%%%%%%%%%%%%%%%%%%%%%%%%%%%%
%%%%%%                                           %%%%%%
%%%%%%          SECTION                          %%%%%%
%%%%%%                                           %%%%%%
%%%%%%%%%%%%%%%%%%%%%%%%%%%%%%%%%%%%%%%%%%%%%%%%%%%%%%%
\section{The bottom horizontal arrows: Fourier--Mukai transforms}\label{sec:FourierMukai}
%%%%%%%%%%%%%%%%%%%%%%%%%%%%%%%%%%%%%%%%%%%%%%%%%%%%%%%
%%%%%%                                           %%%%%%
%%%%%%          SUBSECTION                       %%%%%%
%%%%%%                                           %%%%%%
%%%%%%%%%%%%%%%%%%%%%%%%%%%%%%%%%%%%%%%%%%%%%%%%%%%%%%%
\subsection{Correspondences between sheaves}

We now consider the bottom row of our sets of equivalent data. It is associated to the diagram
\begin{diagram}
&&\PP^1\times T \times T^*\\
&\ldTo<{\pi_1}&&\rdTo<{\pi_2}\rdTo(4,2)>\varphi\\
\PP^1\times T &&&&
\PP^1\times T^* &\rTo_{\psi}& T^*.\end{diagram}
For $a\in\PP^1$, let $D_a$ be the divisor $\{a\}\times T^*$ on $ \PP^1\times T^*$; the $D_a$ are all linearly equivalent. 

We  relate the items on the bottom row of Diagram \eqref{basic-diagram}:
\begin{enumerate}  
\item  Pairs  $(\FFl,\rho) $ on $ T^*$, consisting of a holomorphic vector bundle $\FFl$ and a meromorphic section $\rho$ of its automorphism bundle.  We write $\ch(\FFl ) = k + \ell [\omega_{T^*}] $ (so that it has rank $k$, degree $\ell$) and ask that the determinant of the automorphism $\rho$ has   zeroes at $n$ points $t^{*,+}_j$, and poles at $n$ other points $t^{*,-}_j$. We  assume that these singularities are simple, in the sense that there exists $g_j$, $h_j$ holomorphic and invertible matrix valued functions such that $\rho$ has the local form $g_j(t-t_j^{*,+})\diag((t-t_j^{*,+}),1,1,\ldots,1)h_j(t-t_j^{*,+})$ at $t_j^{*,+}$ and the form  $g_j(t-t_j^{*,-})\diag((t-t_j^{*,-})^{-1},1,1,\ldots,1)h_j(t-t_j^{*,-})$ at $t_j^{*,-}$;   $\rho$ is invertible everywhere else.

We focus on a case of interest to us, when  $\ell = n$;  in addition, we ask that the sheaves be $\widehat\theta^\pm$-stable, in the sense of section 2; the slope $\sll_{\widehat\theta^\pm}(\FFl ,\rho)$ is then  zero.

\item  Sheaves $\KKl $ on $ \PP^1\times T^*$, supported on a curve $C$ in the linear system $[kT^* + n\PP^1]$. The Chern character of $\KKl$ is 
\[\ch(\KKl)= n[\omega_{T^*}] + k [\omega_\PP] + (\ell-n)[\omega_{T^*}\wedge \omega_\PP].\]
The curve $C$ intersects each $ \PP^1\times \{t^*\}$ in  a discrete set of points; it intersects $\{0\}\times T^*$ in the $n$ points $t^{*,+}_j$, and $\{\infty \}\times T^*$ in the $n$ other points  $t^{*,-}_j$, all with multiplicity one. In the neighbourhood of $\{0,\infty \}\times T^*$, the curve $C$ is reduced and intersects $\{0,\infty \}\times T^*$ transversely; the sheaf $\KKl$ is a line bundle over the curve near these points. 

Again, we focus on the case $\ell = n$. We define a $\widehat\theta^\pm$-degree for $(C,\KKl)$ , and its subobjects $(C', \KKl ')$;  for the pair $(C', \KKl ')$, with $C'$ intersecting $z= 0,\infty$ in $n'$ points, and 
$\ch(\KKl ') =  k'[\omega_{T^*}] + n' [\omega_{\PP}] + (\ell'-n') [\omega_{\PP}\wedge \omega_{T^*}]$
the degree becomes
\[\delta^{T^*}_{\widehat\theta^\pm}(C',\KKl') = \ell' - \sum_{t_i^{*,+}\in C'\cap \{z=0\}} \!\!\! \!\!\!  \widehat\theta_i^+ \quad + \sum_{t_i^{*,-}\in C'\cap \{z=\infty\}}  \!\!\! \!\!\! \widehat\theta_i^-.\]
 The pair $(C,\KKl)$ has  $\widehat\theta^\pm$-degree zero; it is stable when there is no subobject $(C',\KKl')$ with positive or zero $\widehat\theta$-degree. We ask that the pair be stable.
\item  Bundles $\EE $ on $ \PP^1\times T$, such that for all $t^*\in T^*$,  there is no non-trivial map for generic $a\in \PP^1$ from $L_{t^*}$  to 
$\EE |_{\{a\}\times  T}$.  The Chern character of $\EE$ is $n  +(\ell-n)[\omega_\PP]  - k [\omega_{\PP}\wedge  \omega_T]$. Over $\{0\}\times T$, the bundle decomposes as the holomorphic sum of degree zero line bundles $\oplus_j L_{t^{*,+}_j}$
over $\{\infty\}\times T$, it decomposes as  $\oplus_j L_{t^{*,+}_1}$. 

Again, we focus on the case for which $\ell-n= 0$, so that $\EE$ has first Chern class zero. One has a notion of parabolic degree, for subbundles $\EE'$ of $\EE$ which are of rank $n'$, first Chern class $d[\omega_\PP]$, and so of first Chern class zero on the $\{a\}\times T$. These must also be sums    $\oplus_{j\in J^+(\EE')}L_{t^{*,+}_j}$ at $a= 0$, and $\oplus_{j\in J^-(\EE')}L_{t^{*,-}_j}$ at $a=\infty$, and  we set, for our parabolic degree
\[\delta_{  \theta^\pm}(\EE') =  d -\sum_{j\in J^+(\EE')} { \theta}^+_i +  \sum_{j\in J^-(\EE')} {  \theta}^-_i.,\]
as in Equation \eqref{eqn:deltathetapm}.
One has $\delta_{\theta^\pm}(\EE) = 0.$
Again, one asks that the bundle be stable, in the sense that any subbundle $\EE'$ must have $\delta_{\theta^\pm}(\EE') < 0.$
\end{enumerate}

The aim of this section is to show the following theorem.
\begin{theorem}\label{holomorphic-equivalences}
These three types of object are equivalent, that is they are related by bijective transforms, under the hypotheses of degree and  stability that we have given.
\end{theorem}

The passage from the first to the second item is a spectral curve construction and the passage from the second to the third is via a Fourier--Mukai transform.

\subsubsection{From pair to sheaf on spectral curve\ok}
 Let $\widetilde\FFl $ be the subsheaf of sections $s$ of $\FFl $ such that $\rho(s)$ is holomorphic; it has degree $\ell-n$. (If $\rk\FFl=1$, then $\widetilde\FFl=\FFl(-t_1^{*,-}-\cdots-t_n^{*,-})$.) Let $D_a$ be the divisor $\{a\}\times T^*$. There are two maps $I, \rho$ from $\psi^*\widetilde{\FFl}(-D_\infty)$ to $\psi^*{\FFl}$; the first is the inclusion, and the second is given by the automorphism $\rho$. One defines $\KKl $ over $\PP^1\times T^*$ by the exact sequence
\begin{equation}\label{FtoK}
\begin{diagram}
0& \rTo& \psi^*\widetilde{\FFl}(-D_\infty)&\rTo^{\rho -z I}& \psi^*{\FFl}&\rTo^\pi&\KKl&\rTo& 0.\end{diagram}
\end{equation}
 The Chern character of $\KKl $ is $n[\omega_{T^*}] + k [\omega_\PP] + (\ell-n)[\omega_{T^*}\wedge \omega_\PP]$;  $\KKl $ is supported on the curve $C$ cut out by $\det(\rho -zI) = 0$, which lies in the linear system $k[ {T^*}] + n [ \PP^1]$. Near the points $t_i^{ *,\pm}$, the normal form of the endomorphism $\rho$ tells us that the curve is reduced, intersects $z=0,\infty$ transversely, with $\KKl$ a line bundle near those points.

We note one feature of the curve $C$ supporting $\KKl $, when $\KKl $ is obtained as above: its fibre for the projection $C\rightarrow T^*$ is everywhere discrete. In other words, $C$  contains no ``horizontal" component of the form $\PP^1\times \{t^*\}$. (On the other hand, ``vertical" components of the form $\{a\}\times T^*, a\neq 0,\infty$ are possible.) 

Subbundles of $\FFl$  invariant under $\rho$ give directly subobjects $(C',\KKl')$ of $(C,\KKl)$, the degrees coincide, and the stability criteria are a direct translation.

\subsubsection{From sheaf on spectral curve to a pair\ok}  
Conversely, starting out with the sheaf $\KKl $, one sets $\FFl  = \psi_*(\KKl )$, and, assuming $z$ is the affine coordinate of $\PP^1$,  $\rho = \psi_*(\times z)$.  This transform $\SSS'\colon \KKl\mapsto (\psi_*\KKl,\psi_*(\times z))$ is the inverse of the   spectral transform $\SSS\colon (\FFl,\rho)\mapsto \KKl$ defined by Equation \eqref{FtoK}, as we now prove.

First, note that $\psi^*\FFl$ is trivial on the $\PP^1$ fibers, so $\psi_*\psi^*\FFl=\FFl$.  On the other hand, $\psi^*\widetilde\FFl(-D_\infty)$ is isomorphic to $\OO(-1)^{\oplus k}$ on the fibers, hence $R^i\psi_*(\psi^*\widetilde\FFl(-D_\infty))=0$ for all $i$.   Hence applying $\psi_*$ to the exact sequence \eqref{FtoK} defining $\KKl$, we obtain $\FFl=\psi_*\KKl$.  Now it only remains to see that $\psi_*(\times z)=\rho$ to prove that 
 $\SSS'\circ\SSS(\FFl,\rho)=(\FFl,\rho)$.  To prove this equation, we note that we have a commuting diagram
 \begin{equation}\begin{diagram} 
\psi^*{\FFl}&\rTo^\pi&\KKl\\
 \dTo<\ \ \rho &&\dTo<\ \ z&\\
\psi^*{\FFl}&\rTo^\pi&\KKl\\
 \end{diagram}  \end{equation}
 since $\pi\circ(\rho-zI) = 0$, and so $\pi\circ \rho = \pi \circ zI= z\pi$. But then on the pushdown, with the isomorphism given by $\pi$, we have
$\psi_*(\times z)=\rho$.

Now we aim to prove that $\SSS\circ\SSS'(\KKl)=\KKl$.
 Consider the sheaf ${\KKl}(-D_\infty)$; there are two maps of ${\KKl}(-D_\infty)$ into $\KKl$; the first is simply the inclusion $i$, while the second is the multiplication by the coordinate $z$. Let $I, \rho$ respectively denote the maps that they induce on direct images:
 \begin{equation} I, \rho\colon \psi_*({\KKl}(-D_\infty)) \rightarrow \psi_*{\KKl}. \end{equation}
 Now lift back; one can define ${\widehat\KKl}$ by
 \begin{equation}
  \begin{diagram}
 0&\rTo& \psi^*\psi_*({\KKl}(-D_\infty))(-D_\infty) &\rTo^{(\rho-zI )}& \psi^*\psi_*{\KKl} &\rTo^{ev}& {\widehat\KKl}  &\rTo& 0.   
 \end{diagram}
\end{equation}
It is then straightforward to see that ${\widehat\KKl}$ is isomorphic to $\KKl$. Indeed, we have that the evaluation/restriction map $ ev\colon \psi^*\psi_*{\KKl} \rTo \KKl$ is zero on $\psi^*\psi_*({\KKl}(-D_\infty))(-D_\infty)$, essentially because $ev(zs) -z\ ev(s)  = 0$. Thus we get a map on the quotient $\mu\colon {\widehat\KKl}\rTo \KKl $. On the other hand, pushing down the defining sequence of ${\widehat\KKl}$ gives an isomorphism $\psi_*{\KKl}\rightarrow \psi_*{\widehat \KKl}$, whose composition $ev_*$ with $\mu_*$ is also an isomorphism. Thus $\mu_*$ is an isomorphism on sections, which, since the support of $\KKl$ is discrete on the fiber, tells us that $\mu $ is also.  Thus $\SSS\circ\SSS'(\KKl)=\KKl$.

\subsubsection{\texorpdfstring{Fourier--Mukai: from sheaf on a curve to bundle on $\PP^1\times T$}{Fourier--Mukai: from sheaf on a curve to bundle on PxT}}\label{sec:FM sheaf on curve to bundle}
Now consider the passage between the last two items on our list. The sheaf $\EE $ is obtained as the Fourier--Mukai transform of $\KKl $, as follows. Let $\Delta $ be the diagonal in $T\times T^*$ obtained by the identification of $\R^2$ and $(\R^2)^*$ given by the symplectic form. Let $\eta_\Delta$ be the Poincar\'e dual of the homology class of $\eta$.  The normalized Poincar\'e bundle on $T\times T^*$ is represented by the divisor $\Delta-\{p_0\}\times  T^*-T\times \{p_0^*\}$ and its restriction to $T\times \{t^*\}$ is the degree zero line bundle $L_{t^*}$.  Let  $\PPP$ be its pullback to  $\PP^1\times T \times T^*$, a line bundle represented by the divisor $\PP^1\times (\Delta - \{p_0\}\times T^* - T \times \{p^*_0\}) $.     We pull back $\KKl $ to $\PP^1\times T \times T^*$ and tensor it with $\PPP$. We then have 
\begin{align*}
	\ch(\pi^*_2(\KKl)\otimes\PPP)&=\Bigl(n[\omega_{T^*}]+k[\omega_{\PP}]+(\ell-n)[\omega_{T^*}\wedge\omega_{\PP}]\Bigr)\Bigl(1+\eta_\Delta-[\omega_T]-[\omega_{T^*}]-[\omega_T\wedge\omega_{T^*}]\Bigr)\\
	&=n[\omega_{T^*}]+k[\omega_{\PP}]+(\ell-n-k)[\omega_{T^*}\wedge \omega_\PP]+k[\omega_\PP]\Bigl(\eta_\Delta-[\omega_T]\Bigr)
	-k[\omega_\PP\wedge\omega_T\wedge\omega_{T^*}]
\end{align*}
and so
\begin{align*}
	(\pi_1)_*\ch(\pi^*_2(\KKl)\otimes\PPP)  = n+(\ell-n)[\omega_{\PP}]-k[\omega_\PP\wedge \omega_T].
\end{align*}

Since the relative tangent bundle of the projection $\pi_1$ is trivial, this cohomology class, by the Grothendieck--Riemann--Roch theorem, is the Chern character of $(\pi_1)_!(\pi_2^*(\KKl )\otimes \PPP)$; there remains the problem of representing this element derived category
 by a single locally free sheaf.

\begin{proposition}
Under the hypotheses governing $(C,\KKl)$, the sheaf $R^1 (\pi_1)_*(\pi_2^*(\KKl )\otimes \PPP) =0$, and so 
$(\pi_1)_!(\pi_2^*(\KKl )\otimes \PPP)= (\pi_1)_*(\pi_2^*(\KKl )\otimes \PPP)$.
\end{proposition}

\begin{proof} When $C$ has no components of the form $\{a\}\times T^*$, this proposition is straightforward: the fiber of the map consists of discrete points, possibly with multiplicity, and there is no $H^1$. One then has  $(\pi_1)_!(\pi_2^*(\KKl )\otimes \PPP)= (\pi_1)_*(\pi_2^*(\KKl )\otimes \PPP)$ which we  then set to be $\EE$; the spaces of sections over the fibers have constant rank, and so $\EE$ is locally free.

One must now deal with the case when there are components of $C$ of the form $\{a\}\times T^*$, if they exist; this can only occur, by our hypotheses, if $a\neq 0,\infty$. The argument is local over $\PP^1$, so one may assume that there is a unique such vertical component.

We have just seen that the datum of $(\FFl,\rho)$ is equivalent to that of $\KKl$. We note that the (generalized) eigenspace associated to $a$ can be ``isolated'' into a subsheaf: indeed, consider the sheaf $\FFl_a$ defined as the kernel of $(\rho-a\id)^m$, for a power $m$ large enough for it to stabilise.

One then has a diagram on $\PP^1\times T^*$:

\begin{equation}\begin{diagram} 
0&\rTo&\psi^*\widetilde\FFl_a(-D_\infty)&\rTo &\psi^*\widetilde\FFl(-D_\infty)&\rTo&\psi^*\widetilde\QQ_a&\rTo&0\\
&&\dTo<\ \ (\rho-z\id)&&\dTo<\ \ (\rho-z\id)&&\dTo&&\\
0&\rTo&\psi^*\FFl_a&\rTo & \psi^*\FFl&\rTo&\psi^*\QQ_a&\rTo&0\\
&&\dTo&&\dTo&&\dTo&&\\
0&\rTo&\KKl_a&\rTo & \KKl&\rTo&\LL_a&\rTo&0\\
 \end{diagram}\end{equation}
 Here the last row are the cokernels of the maps $\rho-z\id$; the sheaf $ \KKl_a$ is supported over $D_a=\{a\}\times T^*$; the sheaf $\LL_a$ is supported on a curve $C'$ which near $D_a$ has discrete fibers over $\PP^1$.

% Now lift this last row to $\PP^1\times T \times T^*$, tensor with $\PPP$, and push down to $\PP^1\times T$ to obtain the exact sequence
% \begin{equation}\begin{diagram}[size=1em]
%0&\rTo&(\pi_1)_*(\pi_2^*\KKl \otimes \PPP)&\rTo &(\pi_1)_*(\pi_2^*\LL_a\otimes \PPP)&\rTo^{\alpha}&R^1(\pi_1)_*(\pi_2^*\KKl_a\otimes \PPP)&\rTo&&R^1(\pi_1)_*(\pi_2^*\KKl\otimes \PPP) &\rTo&0.\\
 %\end{diagram}\end{equation}
%Indeed, we have that $(\pi_1)_*(\pi_2^*\KKl_a\otimes \PPP)= 0$, because it is torsion free and  its support would have to be on $\{a\}\times T$; also $R^1(\pi_1)_*(\pi_2^*\LL_a\otimes \PPP)=0$ near $\{a\}\times T$, because the fibers of its support over $\PP^1\times T$ are discrete.  

We want to show that $R^1(\pi_1)_*(\pi_2^*\KKl\otimes \PPP )=0$. It suffices to show that $H^1(D_A, \KKl\otimes L) = 0$ for any line bundle $L$ of degree zero. One has over $D_a$ that $ \KKl =   \FFl/ (\rho-a\id)(\widetilde\FFl)$.  
In turn, $(\rho-a\id)(\widetilde\FFl)$ is a subsheaf of $\FFl$, and completes to a $\rho$-invariant subbundle $\II $ of $\FFl$  of the same rank (invariant in the sense that $\rho(\II\cap \widetilde\FFl)\subset \II$); $\KKl $ then maps to a bundle $\widetilde\KKl = \FFl/\II $ of the same rank, essentially quotienting out torsion. It suffices that $H^1( \{a\}\times T^*, \widetilde\KKl\otimes L) = 0$.

This vanishing then follows from stability. Indeed:
\begin{lemma} Let 
\begin{equation*}\begin{diagram}[size=1.5em]
0&\rTo& \II &\rTo &  \FFl&\rTo& \widetilde \KKl &\rTo&0
\end{diagram}
\end{equation*}
be our sequence of $\rho$-invariant bundles. Then   $H^1(\{a\}\times T^*,  \widetilde\KKl\otimes L)=0$ for any line bundle $L$ of degree zero.
\end{lemma}

This vanishing can be seen to follow from the stability of $(\FFl,\rho)$, or, what is equivalent, the stability of $(\FFl^*,\rho^*)$.  We note that  $H^1(T^*,  \widetilde\KKl\otimes L)$ is dual to  $H^0(T^*,  \widetilde\KKl^*\otimes L^*)$; the bundle $ \widetilde\KKl^*\otimes L^*$ is an invariant subbundle of $\FFl^*\otimes L^*$, and the weights on this bundle are all zero, as the sheaf $\widetilde\KKl$ is associated to the eigenvalue $a\neq 0, \infty$. Thus the degree of $ \widetilde\KKl^*$ for stability is the ordinary degree, and must be negative; the same holds for any invariant subbundle. On the other hand, a  section of $\widetilde\KKl^*\otimes L^*$ would give a subsheaf of positive or zero degree; it might not be invariant, but iterating $\rho^*$ on it gives a map, generically an isomorphism,  from a bundle of zero degree to an invariant subbundle of $ \widetilde\KKl^*\otimes L^*$, which would then be of non-negative degree, a contradiction. \end{proof}

We thus have  that $(\pi_1)_!(\pi_2^*(\KKl )\otimes \PPP) = (\pi_1)_*(\pi_2^*(\KKl )\otimes \PPP)$, and we set this to be $\EE $.
It  is locally free, by Grauert's theorem, as the fibers have constant rank. On $\{a\}\times T$, near $a=0,\infty$, the bundle $\EE$ is a sum of line bundles of degree zero, essentially determined by the curve $C$ over $a \in \PP^1$. In particular, over $a =0$, the bundle $\EE$ is a sum
$\oplus_j L_{t^{*,+}_j}$; over $\{\infty\}\times T$, it decomposes as  $\oplus_j L_{t^{*,+}_1}$. As these sets of line bundles are supposed disjoint, there is no line bundle of degree zero that maps to the bundles over $a= 0,\infty$, and so for generic $a$, as promised. 

 \subsubsection{\texorpdfstring{Fourier--Mukai: from bundle on $\PP^1\times T$ to a sheaf on a curve}{Fourier--Mukai: from bundle on PxT to a sheaf on a curve}}
Using results of Mukai \cite{FM}, one has an inverse to this operation:

\begin{prop} Suppose that over $\{0\}\times T$, the bundle $\EE$ is a sum
$\oplus_j L_{t^{*,-}_j}$ and that  over $\{\infty\}\times T$, it decomposes as  $\oplus_j L_{t^{*,-}_j}$, with these two sets of line bundles disjoint.
The transform $\KKl \mapsto \EE $ has as inverse: $ \EE  \mapsto  R^1(\pi_{2})_*(\pi_1^*(\EE )\otimes {\PPP^*})$. 
\end{prop}
\begin{proof} Let $\PP^1\times T\times T^* \times T$ have projections $\pi_{12}, \pi_{23},\pi_{13} $ onto $\PP^1\times T\times T^*$, $\PP^1\times T^*\times T$, $\PP^1\times T\times T$, and $\pi_{1}, \pi_{2},\pi_{3} $ onto $\PP^1\times T $, $\PP^1\times T^* $, $\PP^1\times T$.

Lifting to $\PP^1\times T\times T^* \times T$, and applying results of \cite[Section 2]{FM} the composition of the derived functors $\EE \mapsto {\ca R}_*\pi_2 (\pi_1^*\EE \otimes \pi_{12}^*{\PPP^*})$ and $K\mapsto {\ca R}_*\pi_3 (\pi_2^*K\otimes \pi_{23}^*\PPP)$ is given by 
$\EE \mapsto {\ca R}_*\pi_3 (\pi_1^*\EE \otimes {\ca H})$, where $H$ is the ${\ca R}_*\pi_{13}(\pi_{12}^*\PPP\otimes \pi_{23}^*\PPP)$. By the results of section 3 of the same paper, $H$ is the lift to $\PP^1\times T\times T$ of the structure sheaf of the  diagonal $(x, x)$ in $T\times T$. We see that the derived functor  in this case is concentrated in one degree (that is, on $H^1$) and so we  have our sheaf.  

Indeed, the inverse transform, on generic $\{a\}\times T$, essentially gives us the decomposition as a sum of line bundles. Thus, the curve $C$ over $a\in \PP^1$ encodes the line bundles of degree zero which map non-trivially to  $\EE $ over $\{a\}\times T$, so that $C$ intersects $\{0\}\times T^*$ at the points ${t_i^{*,+}}$, counted with multiplicity, and $\{\infty\}\times T^*$ at the points ${t_i^{*,-}}$, again counted with multiplicity. The Chern character of $\KKl$ tells us that these multiplicities need to be one, since the $ {t_i^{*,+}},{t_i^{*,-}}$ are all distinct, as we have supposed. The curve $C$ is  smooth and reduced near $z=0, \infty$, and the sheaf $\KKl$ is of rank one over the smooth locus of $C$.

Note that if one supposes that the sets $\{t_i^{*,-}\}$, $\{t_i^{*,+}\}$ are disjoint, this precludes the possibility of vertical fibres $\PP^1\times \{t^*\}$ in $C$. (In particular, if $\EE  = {\EE '} \oplus L_{t^*_0}$, with $L_{t^*_0}$ pulled back from $T$, then the spectral curve would have a vertical component over $t^*_0$; this corresponds to adding a flat $\U(1)$ connection to an instanton of rank $(n-1)$.) The fact that this decomposition is different at zero, infinity, tells us that the curve contains no lines $\PP^1\times \{t^*\}$. 

 It also forces  $H^0(\PP^1\times T, \EE \otimes L_{t^*}) = 0$, and dually,   $H^2(\PP^1\times T, \EE \otimes L_{t^*}) = 0$; this then means that $H^1(\PP^1\times T, \EE \otimes L_{t^*})$ is of constant dimension in $t^*$, which forces $\FFl $ to be locally free, by Grauert's theorem. In turn, then $\KKl$ is supported on the curve $C$, with support strictly of codimension one. The Chern character of the pushdown tells us that the curve has intersection of degree $n$ on the generic $\{a\}\times T^*$, and the decompositions over $a=0,\infty$ tell us that this intersection is transverse, as advertised, with the   sheaf $\KKl$ being a line bundle on the curve near $a=0,\infty$.\end{proof}

Summarizing, the correspondences give the relation of Chern characters:
\bigskip

\hspace{0.75truein} \label{tableau}\begin{tabular}{|c | c |c |} \hline $(\FFl , \rho)$& $\KKl $& $\EE $\\  \hline 
 $k + \ell[\omega_{T^*}]$ & $k[\omega_{T^*}] + n [\omega_{\PP}] + (\ell-n)[\omega_{\PP}\wedge \omega_{T^*}]$ & $n   +(\ell-n)[\omega_\PP]  - k [\omega_{\PP}\wedge  \omega_T]$\\
  $\det(\rho)$  has $n$ poles&&\\ \hline  \end{tabular}
  \bigskip
  
  We recall that we are specializing to $\ell = n$.

\subsubsection{Correspondences between stability criteria}
We have stability criteria for the first and the third types of data above. Let us see how those criteria correspond.

For the pairs $(\FFl , \rho)$, one only considers $\rho$-invariant subsheaves $\FFl'$;  the restriction $\rho'$ of $\rho$ to $\FFl'$ has a well defined determinant, which has an order $\ord_{t_i^{*,\pm}}(\det(\rho'))$ at the singularities $t_i^{*,\pm}$. If $\ch({\FFl}') = k' +\ell'[\omega_{T^*}]$, we define a degree
\[\delta^{T^*}_{\widehat\theta^\pm}(\FFl', \rho') = \ell' - \sum_{i} \widehat\theta_i^+ \ord_{t_i^{*,+}}(\det(\rho'))+\widehat\theta_i^- \ord_{t_i^{*,-}}(\det(\rho')),\]
 and dividing by the rank $k'$, a slope. The pair is (semi-)stable if and only if for all invariant $\FFl'$, the slope of $\FFl'$ is less (resp. less than or equal to) that of $\FFl $. 

We note that the condition of invariance is quite strong. Indeed, for example, switching perspectives from $(\FFl , \rho)$ to the associated $(C, \KKl )$, there are no proper non-zero invariant subsheaves $\FFl '$ if $C$ is irreducible (it is then necessarily reduced, as it is reduced near $z=0,\infty$). More generally, if $C$ is reduced, invariant subsheaves $\FFl'$ are associated to components $C'$ of $C$; in particular, the order of the determinant of $\rho'$ at $t_i^{*,+}$ is in our case $0$ if $C'$ does not intersect $z=0$ at that point, and $1$ if it does; similarly, at $t_i^{*,-}$, the order of the determinant is  $0$ if $C'$ does not intersect $z= \infty$ at that point, and $-1$ if it does. Thus for the pair $(C', \KKl ')$, with $C'$ intersecting $z= a$ in $n'$ points, and 
$\ch(\KKl ') =  k'[\omega_{T^*}] + n' [\omega_{\PP}] + (\ell'-n') [\omega_{\PP}\wedge \omega_{T^*}]$
the degree becomes
\[\delta^{T^*}_{\widehat\theta^\pm}(\FFl', \rho') = \ell' - \sum_{t_i^{*,+}\in C'\cap \{z=0\}}  \widehat\theta_i^+ \quad + \sum_{t_i^{*,-}\in C'\cap \{z=\infty\}}  \widehat\theta_i^-.\]
 
The stability condition for pairs $(C,\KKl)$ is simply a translation of the one for $(\FFl,\rho)$.

On the other hand, for bundles $\EE $ over $\PP^1\times T$, of degree zero when restricted to the $\{a\}\times T$, with
$\ch(\EE ) =   n   +(\ell-n)[\omega_\PP]  - k [\omega_{\PP}\wedge  \omega_T]$
then one has a stability condition as follows: one looks at all subbundles $\EE '$ with $\ch(\EE ') =   n'   +(\ell'-n')[\omega_\PP]  - k' [\omega_{\PP}\wedge  \omega_T]$. These bundles are    topologically trivial over each $\{a\}\times T$, and near zero or  infinity   in $\PP^1$. If $\EE |_{\{a\}\times T}  = \oplus_iL_{t_i^*(a)}$, then $\EE '|_{\{a\}\times T}$ is given by summing a subset of the $L_{t_i^*(a)}$. In other words, $\EE '$ corresponds to a component $C'$ of the spectral curve. The (parabolic) $\theta^\pm$-degree then becomes
\[\delta_{\theta^\pm}(\EE') = (\ell'-n') - \sum_{t_i^{*,+}\in C'\cap \{a=0\}}  \theta_i^+ \quad +\sum_{t_i^{*,-}\in C'\cap \{a=\infty\}}  \theta_i^-.\]
In the general case, the slopes of our two sets of data are not terribly well related. The reason is that for bundles $\EE $, one is dividing $\delta^T_\theta$ by $n'$, while for $\FFl '$, one is dividing $\delta^{T^*}_\theta$ by $k'$. On the other hand, in the case which is of interest to us, of bundles associated to instantons, one has $\sum_i \theta_i^+=\sum_i\theta_i^-$, as well as $\ell = n$. The (semi)-stability condition for $\EE $ reduces to 
\[\delta_{\theta^\pm}(\EE') = (\ell'-n')  - \sum_{t_i^{*,+}\in C'\cap \{z=0\}}  \theta_i^+ \quad +\sum_{t_i^{*,-}\in C'\cap \{z=\infty\}}  \theta_i^- \quad < (\leq)\quad 0.\]
Now since $\hat\theta_i^+= \theta_i^+ +1, \hat \theta_i^-= \theta_i^- $, this condition becomes 
\[\delta_{\theta^\pm}(\EE') =  \ell' - \sum_{t_i^{*,+}\in C'\cap \{z=0\}}  \hat\theta_i^+ \quad +\sum_{t_i^{*,-}\in C'\cap \{z=\infty\}}  \hat\theta_i^- \quad < (\leq) \quad 0.\]
 
Since $\delta^{T^*}_{\widehat\theta^\pm}(\FFl ,\rho)= 0$ (see the data in the table of page \pageref{tableau}), we can compare this criterion criterion to the one for $(\FFl,\rho)$:
\[\delta^{T^*}_{\widehat\theta^\pm}(\FFl', \rho') = \ell' - \sum_{t_i^{*,+}\in C'\cap \{z=0\}}  \widehat\theta_i^+ \quad +\sum_{t_i^{*,-}\in C'\cap \{z=\infty\}}  \widehat\theta_i^- \quad < (\leq) \quad 0.\]

\begin{proposition}
The bundle $\EE $ associated to an instanton is $\theta^\pm$-stable iff the associated pair $(\FFl , \rho)$ or the pair $(C,\KKl)$ is $\widehat\theta^\pm$-stable.
\end{proposition}

This concludes our set of equivalences of the holomorphic data;   Theorem \ref{holomorphic-equivalences} is proven.

%%%%%%%%%%%%%%%%%%%%%%%%%%%%%%%%%%%%%%%%%%%%%%%%%%%%%%%
%%%%%%                                           %%%%%%
%%%%%%          SECTION                          %%%%%%
%%%%%%                                           %%%%%%
%%%%%%%%%%%%%%%%%%%%%%%%%%%%%%%%%%%%%%%%%%%%%%%%%%%%%%%
\section{The top horizontal arrow: the Nahm transform}\label{sec:NahmTransform}

We are concerned in this section with the top row of our diagram \ref{basic-diagram}. We  examine this row in the instanton to monopole direction, with some comments at the end on the inverse direction.
More specifically, the Nahm transform  gives  us an equivalence between 
\begin{itemize}
\item $(E,\nabla)$, a rank $n$ bundle with a finite energy charge $k$ anti-self-dual $\U(n)$ connection on the flat cylinder  $\RT$. At $+\infty$,   $(E,\nabla)$ is asymptotic to a fixed flat $\U(1)^n$ connection on $S^1 \times T $ which, on each $ \{\mu\}\times T$, corresponds to the sum of line bundles $\oplus_i L_{t^{*,+}_i}$, and on the $S^1$ factors, acts on the $L_{t^{*,+}_i}$ by $\frac{\partial}{\partial \mu} +  i\theta^+_i$. One has, at $-\infty$, the same, but with $(t^{*,-}_i, \theta^-_i)$.   The set  $\{t^{*,+}_1,\ldots,t^{*,+}_n, t^{*,-}_1,\ldots,t^{*,-}_n\}$ contains $2n$ distinct points.

\item $(F,\nabla,\phi)$,  a $\U(k)$ monopole, defined on $ S^1\times T^*$, with Dirac type singularities of weight    $(1,0,\ldots,0)$ at $n$ points $(  \theta^+_i, t^{*,+}_i)\in  S^1\times T^*$, and of weight    $(-1,0,\ldots,0)$ at $n$ points $(\theta^-_i, t^{*,-}_i)\in  S^1\times T^*$. 

\end{itemize}

We show in this section that the objects involved in the transform can also be given in a round about way through our basic diagram \eqref{basic-diagram}, with the same results, by going through the holomorphic data. This alternative allows us, for example, to obtain in a fairly straightforward way the bijective nature of the transform, which is more difficult to obtain   via the direct analytical route.
 
 \subsection{From $E$ to $F$; definitions\ok}

This direction has been examined in detail in the paper \cite{benoitpaper}, for the case of rank two (see in particular \cite[Sec.~3]{benoitpaper} for a concise description of Nahm's heuristic). We recall how the the transform works. We  consider the kernel of a Dirac operator $\DD^*_{\theta,t^*}$ in the instanton background, shifted by the central characters  $(-i\theta, -it^*), \theta \in S^1, t^*\in T^*$.  We realise the spin bundles as the product of the  quaternions with $\bbr\times S^1\times T$, with $\sigma_1,\sigma_2,\sigma_3$ acting by $i,j,k$. With these conventions, and writing $t^*$ in orthonormal coordinates as a pair of real numbers $(t_1^*, t_2^*)$, one is considering the Dirac operators
\begin{equation}
	\DD^*_{\theta,t^*} := \DD^*_\nabla-i\sigma_3\theta  -i\sigma_1t_1^* -i\sigma_2t_2^*\colon \Gamma(\RT;S^-\otimes E)\to\Gamma(\RT;S^+\otimes E)\label{eqn:def-Dirac}
\end{equation}
acting on sections  the tensor product of $E$ with the spin bundle $S^-$, and more importantly the $L^2$-index bundle parametrized by $S^1\times T^*$.  

The Weitzenbock formula
\[\DD^*_\nabla \DD_\nabla=\nabla^*\nabla+F_\nabla^+\]
guarantees that $L^2\text{-}\ker(\DD_\nabla)=\{0\}$ when $\nabla$ is ASD, so the index bundle is in fact a bundle of kernels.  That we indeed have a bundle depends crucially on the invariance of the index for a continuous family of Fredholm operators.  For a fixed  $t^*\in T^*$, as long as the limit bundles over the ends $\{\pm\infty\}\times S^1\times T$ have no trivial summand, the operator is indeed Fredholm; see \cite{APS1}. The trivial summands occur when 
\[(\theta,t^*)\in W:= \{(\theta_1^\pm,t_1^{*,\pm}),\ldots,(\theta_n^\pm,t_n^{*,\pm})\}.\]
Elsewhere, that is on $(S^1\times T^*)\setminus W$, one has a  bundle $F$ whose fiber at $(\theta,t^*)$ is
\[F_{(\theta,t^*)}:=L^2\text{-}\ker(\DD^*_{\theta,t^*}).\]

One then defines a connection $\nabla$ and a Higgs field $\phi$ on $E$ by $L^2$ projection $P\colon \underline{L^2(\RT,E\times S^-)}\rightarrow F$ of the corresponding trivial operators $d ,  \times s, s\in \R$ on the trivial bundle with fiber  $L^2(\RT,E\times S^-)$  over $  S^1\times T^*$, that is
\[\nabla = P\circ d,\quad\phi =  i P\circ (\times s).\]

\subsection{From $E$ to $F$: Indices}
One would like to know the rank of $F$, as well as  its degree. We  restrict our attention for the moment to the case $\theta $ constant, and  simply take $\theta = 0$, that is we consider the family of Dirac operators
\[\DD_\nabla^* -i \sigma_1 t^*_1 -i \sigma_2t^*_2\]
parametrized by $t^*=(t^*_1,t^*_2)\in T^*$. This bundle is obtained by taking the standard fiberwise flat unitary connection over the Poincar\'e line bundle $P$ over $ T\times T^*$, lifting it as $\pi_{34}^*P$ to $\R\times S^1\times T\times T ^*$, then taking the tensor product $\pi_{34}^*P\otimes \pi_{123}^*E$, equipping it with the fiberwise tensor product connection, and taking the index bundle of that over $T^*$.

We begin with  a few homotopies and a gauge transformation.  These operations preserve the Fredholmness of the Dirac operators.
\begin{itemize}
\item Modify the connection so that the exponentially decreasing term in the asymptotic form above is in fact zero.  This modification allows us to assume that the support of the curvature of $\nabla$ is compact and contained in $(-s_0, s_0)\times S^1\times T$.
\item We can move the $t^{*,\pm}_i$ to a constant, say $0$.
\item We can move the $\theta^{\pm}_j$ to constants $ m^{\pm}_j+\frac12$, for integers $m^{\pm}_j$, in a way that has $\theta^{\pm}_j$ differing from $m^{\pm }_j+\frac12$ by less than $\frac12$ (so that $m^{\pm}_j$ is the integer part $\floor{\theta^{\pm }_j}$ of $\theta^{\pm }_j$). 
\item Applying  gauge transformations  of the form $\exp (i \mu \diag (m^{-}_j))$, we can set all the $m^{-}_j$ to zero, shifting the ones at $+ \infty$ to $\hat m^{+}_j  = m^{+}_j-m^{-}_j$. Let  $M= \sum_j \hat m^{+}_j=\sum_{j} \floor{\theta_j^+}-\floor{\theta_j^-}=-\widehat\Theta$.
\end{itemize}

We apply the excision principle for indices (see for instance \cite[Appendix B]{benoitthesis}),
 for the following spaces, equipped with bundles and connections. We begin with  line bundles. First pick $s_1>s_0$. We let
\begin{itemize}
\item $L_+$ be the  line bundle with  connection  $f(s) d\mu $ on  $(-s_1, s_1)\times S^1\times T$, flat with  connection $ \frac i2 d\mu $  on $(-s_1, -s_0)$, and flat with connection $\frac{i(2m+1)}2 d\mu $ on $(s_0, s_1)$,
\item $L_- $ be the  flat line bundle with  connection $\frac i2 d\mu $  on $ (-s_1, s_1)\times S^1\times T$,
\item $L^+$ be the  flat line bundle with  connection $\frac i2 d\mu $  on $ \Bigl((-\infty, -s_0)\cup (s_0,\infty)\Bigr)\times S^1 \times T$, and
\item $L^-$ be the  flat line bundle with  connection $\frac i2 d\mu $  on $ (-s_1, s_1)\times S^1\times T$.
\end{itemize}
Now we do some glueings. 
\begin{itemize}
\item Glue $L_+$ to $L^+$, to obtain $L_+^+$, a  line bundle  with connection on  $\RT$, flat with  connection $ \frac i2 d\mu $ on $(-\infty, -s_0)$, and flat with connection $\frac{i(2m+1)}2d\mu$ on $(r_0, \infty)$. On the negative interval, the glueing is trivial; on the positive interval, one needs to glue via a bundle automorphism.
\item Glue $L_-$ to $L^+$, to obtain $L_-^+$, a line bundle  with flat connection $ \frac i2 d\mu $ on  $\RT$.
\item Glue $(-s_1, s_1)$ to $(-s_1, s_1)$ so as to obtain a circle, identifying $(-s_1,-s_0)$ to $(s_0, s_1)$ and $(s_0,s_1)$ to $(-s_1, -s_0)$  at the ends of the interval; on the bundle level, glue $L_+$ to $L^-$, to obtain a line bundle $L_+^-$ over $ S^1\times S^1\times T$ with  connection $\frac i2d\mu $ on $  \Bigl(S^1\setminus  (-s_0,s_0)\Bigr)\times S^1\times T$. This glueing  requires a bundle automorphism at one of the two ends of the circle, and introduce a first Chern class $-m[\omega_{S^1\times S^1}]$ for the resulting bundle on $ S^1\times S^1\times T$.\label{L_+^-}
\item In the same way, glue $L_-$ to $L^-$ to obtain a   flat line bundle $L_-^-$ with  connection $\frac i2d\mu $ on $ S^1\times S^1\times T $.
\end{itemize}

With these constructions, twisting now by our family of connections in $T^*$,
one has for  the index bundles over $T^*$
\[\Ind(L_+^+)\oplus \Ind(L_-^-)\ominus \Ind(L_-^+)\ominus \Ind(L_+^-)=0.\]
Now consider the Chern characters of these bundles. 
\begin{itemize}
\item For $\Ind(L_-^+)$, one sees explicitly that there is no kernel or cokernel for any of the twists; the index bundle is trivial.
\item \label{index theorem families}  For $\Ind(L_-^-), \Ind(L_+^-)$, one can apply the index  theorem for families for bundles over a compact manifold; we note that the first Chern class of $L_+^-$ is $-m[\omega_{S^1\times S^1}]$; for $L_-^-$ it is zero. Using $\ch(P) =1+\eta_\Delta -[\omega_T]-[\omega_T^*]- [\omega_T\wedge\omega_{T^*}]$, one has that $\ch(\Ind(L_-^-))= 0, \ch(\Ind(L_-^+)) =  m[\omega_{T^*}]$.  
\end{itemize}
Combining these facts, we obtain the result we want, namely that 
\begin{equation} \ch(\Ind(L_+^+)) = m [\omega_{T^*}].\label{Abelian}\end{equation}
We now apply our excision again. Take 
\begin{itemize}
\item $E_+$=  vector  bundle with  connection $A$ on  $(-s_1, s_1)\times  S^1\times T$, flat with  connection $\frac i2\id d\mu $ on $(-s_1, -s_0)$, and flat with connection $i \diag(\frac{2\hat m_j+1}2 )d\mu $ on $(s_0, s_1)$;
\item $E_- $=  vector  bundle with  Abelian connection $\diag(f_i(s))d\mu$ on  $ (-s_1, s_1)\times S^1\times T$, flat with  connection $\frac i2\id d\mu $ on $(-s_1, -s_0)$, and flat with connection $i \diag(\frac{2\hat m_j+1}2 )d\mu $ on $(s_0, s_1)$;
\item $E^+$ =  flat line bundle with  connection $\frac i2\id d\mu $ on $ \Bigl((-\infty, -s_0)\cup (s_0,\infty)\Bigr)\times S^1\times T$;
\item $E^-$ =  flat line bundle with  connection $\frac i2\id d\mu $ on $ (-s_1, s_1)\times  S^1\times T$;
\end{itemize}

As above, glue these to obtain bundles plus connection $ E_+^+, E_-^+$ on $  \RT$, and $E_+^-, E_-^-$ on $  S^1\times S^1\times T$; one has $c_1(E_+^-) = -M[\omega_{S^1\times S^1}], c_2(E_+^-) = k[\omega_{S^1\times S^1}\wedge \omega_T]$, while $c_1(E_-^-) =-M[\omega_{S^1\times S^1}], c_2(E_-^-) = 0$, Our index calculation over the compact manifolds for the Dirac operator  for the family of twists parametrised by $T^*$ is $\ch(\Ind(E_+^-)) = -k   +M [\omega_{T^*}], \ch(\Ind(E_-^-)) =  M [\omega_{T^*}]$;  the previous calculation done in the Abelian case above applied to  $\ch(\Ind(E_-^+))$ gives $M [\omega_{T^*}]$. Combining this result with the excision principle, we obtain
\[\ch(\Ind(E_+^+)) = -k  + M [\omega_{T^*}].\]

We thus have the index bundle $ F= -\Ind(E_+^+)$ for our fixed values of $\theta^{\pm }_j$; its first Chern class is the sum 
$\widehat\Theta=\sum_i -\floor{\theta^{+ }_i]+[\theta^{- }_i}$. For different reasons, as we have seen, the sum  $\sum_i -\theta^{+ }_i+\theta^{- }_i$ is also an integer, in fact zero; it does not, however, have to be the same integer.

We also note that the full Nahm transform involves a shift of the $\theta^{\pm }_j$ to $\theta^{\pm }_j + \theta$, so that one is considering the index of
$\DD_{\theta,t^*}$; let ${  F}_\theta$ be the index bundle over $T^*$, as $\theta$ varies. When one of the $\theta^{\pm }_j + \theta$ is an integer, there is a point on the dual torus $T^*$ at which the operator is non-Fredholm, and so the index bundle is not defined at this point; passing through this value of $\theta$, the first Chern class jumps, negatively if it is a $\theta^{+ }_j - \theta$ that is passing through an integer, and positively if it is a  $\theta^{- }_j + \theta$: the degree is the sum  $\sum_i -\floor{\theta^{+ }_i +\theta}+\floor{\theta^{- }_i+ \theta}$. As $\theta$ varies from $0$ to $1$, all the $\theta^{\pm }_j + \theta$ pass in turn through an integer value, and the degree of ${ F}_\theta$ comes back to the initial value.

The family of bundles ${  F}_\theta$ has a natural notion of degree $\delta({  F})$, which is the average degree of ${ F}_\theta$ as one moves $\theta$ from $0$ to $1$; this notion is used in \cite[Definition 3.7]{benoitjacques3} to define an appropriate notion of stability for the pair $({ F},\psi)$. From the easily seen fact
\[\int_0^1 \floor{\alpha +\theta} d\theta = \alpha,\]
one has that, while the degree of $F$ along $\theta$ = constant is $\sum_i -\floor{\theta^{+ }_i +\theta }+\floor{\theta^{-}_i+\theta }$, the average degree over the three-fold is
\[\delta({F}) = \sum_i (-\theta^{+ }_i+\theta^{-}_i) = 0.\]

We have thus defined on the complement of our singular points  a bundle $F$ equipped with a connection and Higgs field, and computed its degree and rank; we know that it satisfies the Bogomolny equation.  The results of \cite{benoitpaper} tell us that the Higgs field has asymptotic behaviour $\frac{i}{2R}\diag(\pm1,0,\ldots,0)+O(1)$ at $R\to 0$ but doesn't prove that $\nabla(R\phi)=O(1)$. The next section, in addition to proving the commutativity of \eqref{basic-diagram}, tells us that we have actual Dirac type singularities.

%%%%%%%%%%%%%%%%%%%%%%%%%%%%%%%%%%%%%%%%%%%%%%%%%%%%%%%
%%%%%%                                           %%%%%%
%%%%%%          SUBSECTION                          %%%%%%
%%%%%%                                           %%%%%%
%%%%%%%%%%%%%%%%%%%%%%%%%%%%%%%%%%%%%%%%%%%%%%%%%%%%%%%

\subsection{From $E$ to $F$: Nahm transform versus the Fourier--Mukai transform}\label{sec:NTvsFM}
We have a diagram of operations \eqref{basic-diagram} with the top row a Nahm transform taking us between monopole and instanton, and the lower row two holomorphic transforms. The vertical correspondences are Hitchin--Kobayashi correspondences. One would like to see that this diagram commutes. We establish this fact  by seeing that starting with an instanton $(E,\nabla)$ on $\RT$,  the pairs $(\FFu,\rho)$,$(\FFl ,\rho)$ corresponding to it via the upper route and the lower route are the same, at least up to some twists.

Let us start with the upper route.  The kernel $\FFu_{\theta,t^*} $ of the Dirac operator $\DD^*_{\theta,t^*}$ of Equation \eqref{eqn:def-Dirac} can be thought of as the space of harmonic sections ${\ca H}^1$ of a Dolbeault  complex
\begin{equation}\label{Dirac}
\begin{diagram}
 L^2(E)&\rTo^{\begin{pmatrix} \nabla_\theta^{0,1}\\ \nabla_{t^*}^{0,1}\end{pmatrix}} & L^2(E \otimes \Lambda^{0,1}( \RT))
  &\rTo^{\begin{pmatrix} \nabla_{t^*}^{0,1}, -\nabla_\theta^{0,1} \end{pmatrix}}&  L^2(E\otimes \Lambda^{0,2}( \RT)).\end{diagram}\end{equation}
Here $\nabla_{t^*}^{0,1}$ is the $\bar\partial$-operator given by the connection in the $T$-direction, shifted by the character $-it^*$, and $\nabla_{\theta}^{0,1}$ is the $\bar\partial$-operator given by the connection in the $\R\times S^1\subset \PP^1$-direction, shifted by the character $-i\theta$. The harmonic sections of course get identified with the first cohomology  of the complex, which is a holomorphic object. The bundle $\FFu$ is given by the restriction to $\theta = 0$, and so  ${\FFu}_{t^*} = H^1_b( \RT, E\otimes L_{-t^*})$, where the subscript $b$ refers to those cohomology classes satisfying the appropriate boundary conditions, and $ L_{-t^*}$ is the flat line bundle on $T$ associated to the character $-it^*$. 
In other words, one lifts the bundle with connection to $\R\times S^1\times T \times T^*$, tensors by the dual of the Poincar\'e bundle $\PPP$ equipped with its natural flat connection, and considers the projection $\phi= \psi\circ \pi_2$ to $T^*$; one is taking the direct image, $R^1\phi_*(\pi_1^* E\otimes \PPP^*)$.
  
For the lower route, one first extends to a holomorphic bundle $\EE$ on $\PP^1\times T$, then $\FFl=\psi_*R^1(\pi_2)_*(\pi^*_1\EE\otimes \PPP^*)$.
Thus $\FFl _{t^*}$ is given by the first cohomology of the complex
\begin{equation}\begin{diagram}
\Omega^{0,0}( \PP^1  \times T, \EE \otimes L_{-t^*})&\rTo^{\db} &\Omega^{0,1}( \PP^1  \times T, \EE \otimes L_{-t^*}) &\rTo^{\db}&\Omega^{0,2}( \PP^1  \times T, \EE \otimes L_{-t^*})\end{diagram}\label{Dolbeault}.\end{equation}
Almost the same complex, with different boundary behaviours at $z=0,\infty$.

\emph{Identifying the bundles}.  There is a natural way of identifying the two cohomology bundles $F$ and $\FFl$, at least over the points $t^*$ where the map $\rho$ is holomorphic and invertible; the reason, as we shall see, is that for a given $t^*$ the cohomology can localise onto a compact set on the cylinder, so that the $L^2$ boundary condition and the compactification boundary conditions coincide. For $\FFl$, the Fourier--Mukai transform splits into two steps, so that one has $\KKl  = R^1(\pi_2)_*(\pi_1^*\EE  \otimes \PPP^*)$, and $\FFl  = \psi_*(\KKl )$. One can do the same for  $E$, setting $\KKu = R^1(\pi_2)_*(\pi_1^*E\otimes \PPP^*)$; away from   $z= 0,\infty$, the sheaf $\KKu$ is supported on the same curve $C$ as $\KKl $; indeed, away from $z=0,\infty$, $\KKu$ and $\KKl $ are naturally identified, and so $\FFu= \psi_* \KKu$ and $\FFl = \psi_*\KKl $, away from the points of intersection of $C$ with $z=0,\infty$.

For later use, we would like to build cocycles representing local sections of $\KKu$ and $\KKl $  fairly explicitly, giving  Dolbeault representatives that work for both cohomologies, on a neighbourhood of a smooth point of the curve $C$.  First, we build a cocycle on $T^* \times T$ for the first direct image of the Poincar\'e bundle, supported  over an open set $D \times T$, where $D$ is a small disk centred on the origin  in $T^*$. Note that since $H^1(T,L)=0$ for any non-trivial $L$ of degree 0, this cocycle can be a bump form supported near $0\in T^*$. Write our elliptic curve $T$ as a quotient $\C^*/(t\rightarrow \nu t)$.  Let $t^*$ belong to our small disk $D$ containing the origin  in $T^*$. We can define line bundles $L_{t^*}$ by saying that sections of $L_{t^*}$ are functions over $\C^*$ satisfying $f(\nu t) = (\nu\overline \nu)^{t^*} f(t)$; the ${t^*}$ can be complex, and ${t^*} = 0$ corresponds to the trivial bundle. We  build a cocycle over $D\times T$. Consider the function $r^{2{t^*}} = (t\overline t)^{t^*}$ on $\C^*$. This function is a section of $L_{t^*}$; now note that $\frac{\partial }{\partial \overline t}(r^{2{t^*}})d\overline t = {t^*}(t\overline t)^{({t^*}-1)} (td\overline t)$. We consider the $L_{t^*}$-valued $(0,1)$-form
\[\mu (t) =  (t\overline t)^{({t^*}-1)} (td\overline t) = \frac{1}{{t^*}}\frac{\partial (r^{2{t^*}})}{ \partial \overline t}d\overline t. \]
Note that for ${t^*}$ away from  the origin, $\mu$ is a coboundary in $t$. Now take a bump function $\phi({t^*})$ supported in $D$, with $\phi= 1$ on an open set containing the origin. Consider on $D\times T$ the compactly supported form
\[ M({t^*}, t) = \phi({t^*}) \mu +\frac{\partial }{\partial {\overline {t^*}}}(\phi({t^*}))\frac{(r^{2{t^*}})}{{t^*}}d\overline {t^*}. \]
This $M$ is
our generator for the first direct image of the Poincar\'e line bundle $\PPP^*$, projecting from $D\times T$ to $T^*$.  

We now use this to write out an explicit Dolbeault cocycle for the first direct image of 
$\EE\otimes \PPP^*$, away from $z=0,\infty$. This direct image $\KKl$ is supported on the curve $C$. For a family of points $(z(\iota), t^*(\iota)) $ on the smooth locus of $C$ , one has that  there   exist local sections $\sigma(\iota)$ of $\EE\otimes \PPP^*$: the space $H^0(\{z\}\times T, \EE\otimes L^*_{-t^*})$ is non-zero if and only if  $H^1(\{z\}\times T, \EE\otimes L^*_{-t^*})$ is. 
 One then has that 
 \begin{equation}\label{cocycle}
\sigma(\iota) \cdot M({t^*(\iota)},t) 	
 \end{equation}
 represents sections both of $\KKu$ and $\KKl $.    Thus, projecting out the $\PP^1$ factor, away from the $t^{*,\pm}_i\in T^*$ over which $C$ intersects $z=0,\infty$, one then has a natural identification of $\FFu= \psi_* \KKu$ and $\FFl = \psi_*\KKl $, away from the points of intersection of $C$ with $z=0,\infty$. 

\emph{Identifying the endomorphisms}. We next want to see that the endomorphisms $\rho$ are identified also. From the Nahm transform side, the endomorphism is given at $t^*$ by solving $(\nabla_\theta - i\phi)(S(\theta,t^*))= 0$ around the circle $\theta\in [0,1]$, and defining $\rho(S(0,t^*)) = S(1,t^*)$. On the cohomology of the complex \eqref{Dirac}, the operator $\nabla_\theta - i\phi$ is  represented on cocycles by $\partial_\theta +s$.
If one has $\Xi(s,\mu, t) $ representing a cocycle for a fixed $t^*$ and for $\theta=0$, then $S(\theta,t^*)=\exp(-\theta s)\Xi(s,\mu, t) $ solves  $(\nabla_\theta -i\phi)(S(\theta,t^*))= 0$. At $\theta = 1$, one then has the cocycle  $\exp(-s)\Xi(s,\mu, t) $, representing the cocycle; one notes, however, that the identification of the cohomology at $\theta= 0$ and $\theta = 1$, so that one is integrating over a circle, and not over an interval, involves an additional shift by a factor $\exp (-i \mu)$, so that the full monodromy of the cocycle is by  $\Xi(s,\mu, t) \mapsto \exp(-s-i\mu)\Xi(s,\mu, t)=z\Xi(s,\mu, t) $. Hence the monodromy is given by $[\Xi(s,\mu, t)] \mapsto [z\Xi(s,\mu, t)] $, or, in other words, by multiplying by the complex coordinate $z$. Thus on $\KKu\cong \KKl$ over $z\neq 0$ and $z\neq \infty$, one has the same operation on fibers $\psi^{-1}(t^*)$ of multiplication by $z$. Taking direct images, the two versions of $\rho$ are therefore identical.

\emph{Behaviour at infinity.} We thus have $(\FFu,\rho)$ and $(\FFl , \rho) $ identified away from the points $t^*_i$ over which the spectral curve $C$ intersects $z= 0,\infty$ in $\PP^1$. One can ask what happens at these points; the answer, basically, is that one is a Hecke transform of the other; equivalently, on the level of sheaves over $\PP^1\times T^*$, $\KKu$ is obtained from $\KKl $ by twisting by a divisor whose support is the set of points $(0, t_i^{*,+}), (\infty, t_i^{*,-})$, and whose coefficients depend on the values of $\theta_i^\pm$. We recall that the (ordinary) degree of $F$ is the sum $\sum -\floor{\theta_i^+}+\floor{\theta_i^-}$, while the (ordinary) degree of $\FFl $ is $n$. 
 
 The comparison between $\KKu$ and $ \KKl $ revolves around whether a cocycle lying  in $L^2$ also lies in our compactification, and vice versa. Let us look at the situation around $z=0$. Here, the spectral curve is, locally,  a graph of a function $t^*_i(z)$ which has a finite limit $t_i^{*,+}$ as $z$ tends to zero. One then has, in a holomorphic trivialisation (corresponding to $\KKl $), as in \eqref{cocycle}, cocycles $c= \sigma(z,t) \cdot M({t^*},t)$. These are finite at $0$. Let us now consider what happens when one goes to the unitary trivialisation, which is related to the holomorphic one by a transition function $(z\bar z)^{-\theta_i^+/2}$. The cocycle, instead of being of order zero at infinity,  now has order  $-\theta_i^+$; to get a non-zero holomorphic $L^2$ section, one must multiply it by $z^{\floor{\theta_i^+}+1}$; in other words, near the singular point $q_i^+ = (0, t_i^{*,+})$,  $\KKu = \KKl  ((-\floor{\theta_i^+}-1) q_i^+)$. If one considers the relation between $\KKu	$ and $ \KKl $ near $q_i^- = (\infty, t_i^{*,-})$, one has in a similar vein that, locally, $\KKu = \KKl  (\floor{\theta_i^-} )$. Thus, globally \[\KKu = \KKl \bigl( - \sum_i (\floor{\theta_i^+}+1) q_i^+ - \floor{\theta_i^-}q_i^-)\bigr).\]
 Correspondingly, the bundles $\FFu$ and $\FFl$ are related by a Hecke transform; as the transform is on eigenspaces of $\rho$, the endomorphism $\rho$ remains well defined on both bundles.

{\it Stability.} We recall the relevant definitions of stability for  $(\FFu,\rho)$ and $(\FFl , \rho)$. In both cases, we have a degree for invariant subsheaves. For $\FFu$, the appropriate notion of degree is  defined using
\[\widetilde\theta_i^+ = \theta_i^+- \floor{\theta_i^+},\quad \widetilde\theta_i^- = \theta_i^- - \floor{\theta_i^-},\]
and for $\FFl$, by 
\[\widehat\theta_i^+ = \theta_i^+ + 1,\quad  \widehat\theta_i^- = \theta_i^-  .\]
 For subsheaves, one sets
by 
\[\delta^{T^*}_{\widetilde\theta}(\FFu', \rho') = \ell' - \sum_{q^+_i\in C'\cap \{z=0\}}  \widetilde\theta_i^+ \quad +\sum_{q^-_i\in C'\cap \{z=\infty\}}  \widetilde\theta_i^-\]
for subsheaves $(\FFu', \rho')$ of $(\FFu,\rho)$, and 
\[\delta^{T^*}_{\widehat\theta}(\FFl', \rho') = \ell' - \sum_{q^+_i\in C'\cap \{z=0\}}  \widehat\theta_i^+ \quad +\sum_{q^-_i\in C'\cap \{z=\infty\}}  \widehat\theta_i^-.\]
for subsheaves $(\FFl ', \rho')$ of $(\FFl , \rho)$. One has $\delta^{T^*}_{\widetilde\theta}(\FFu, \rho) = \delta^{T^*}_{\widehat\theta}(\FFl , \rho) = 0$, and if $(\FFu', \rho')$ corresponds to $(\FFl ', \rho')$ via the Hecke transform which relates $\FFu$ and $\FFl $, then $\delta^{T^*}_{\widetilde\theta}(\FFu', \rho') = \delta^{T^*}_{\widehat\theta}(\FFl ', \rho')$. Therefore

\begin{proposition} 
Thus $(\FFu,\rho)$ is $\widetilde\theta$-(semi)-stable iff $(\FFl , \rho)$ is $\widehat\theta$-(semi)-stable.
\end{proposition}

The proof is a straightforward unwinding of the definitions.

\emph{The Nahm transform gives us a monopole}. Let us consider what this equivalence gives us. The roundabout complex route takes us from $E$, to $\EE$, to $(C,\KKl)$, to $(C,\KKu)$ and then to $(\FFu,\rho)$.
The latter, from \cite{benoitjacques3} is the holomorphic data of a singular monopole $(\widehat F,\widehat \nabla, \widehat \phi)$ on $S^1\times T^*$. On the other hand, the Nahm transform gives us a solution 
$( F, \nabla,   \phi)$ with the same singularities, and with the same holomorphic data $(\FFu,\rho)$. On the monopoles, this gives us an identification of $\widehat F$ and $F$ which intertwines the ``complex" part of the connections, that is $\partial^{0,1}_{T^*}$ and $ \nabla_\theta - i\phi$, and which is bounded at the singularities. The Bogomolny equations on $S^1\times T^*$ are just reductions of the anti-selfduality equations on   $S^1\times  S^1\times T^*$, and an equivalent statement would be that the lifts of $\widehat F$ and $F$ are holomorphically equivalent. One has however for ASD connections the Weitzenb\"ock identity $\nabla_A^*\nabla_A = (\partial^{0,1}_A)^*\partial^{0,1}_A$, and so the identification is flat, that is it intertwines the connection. This then shows that $(F,\nabla,\phi)$ is one of our desired singular monopoles.

We can now summarize our circle of equivalences.

\begin{theorem} One has the following equivalent data:

\begin{itemize}
\item $(E,\nabla)$ is a rank $n$ bundle with a finite energy charge $k$ anti-self-dual $\U(n)$ connection on the flat cylinder  $S^1\times \R \times T$. At $+\infty$,   $(E,\nabla)$ is asymptotic to a fixed flat $\U(1)^n$ connection on $S^1 \times T $ which, on each $ \{\mu\}\times T$, corresponds to the sum of line bundles $\oplus_i L_{t^{*,+}_i}$, and on the $S^1$ factors, acts on the $L_{t^{*,+}_i}$ by $\frac{\partial}{\partial u} +  i\theta^+_i$. One has, at $-\infty$, the same, but with $(t^{*,-}_i, \theta^-_i)$.
\item $(F,\nabla,\phi)$ is a $\U(k)$ monopole, defined on $ S^1\times T^*$, with Dirac type singularities of weight    $(1,0,\ldots,0)$ at $n$ points $(  \theta^+_i, t^{*,+}_i)\in  S^1\times T^*$, and of weight    $(-1,0,\ldots,0)$ at $n$ points $(\theta^-_i, t^{*,-}_i)\in  S^1\times T^*$. The points $t^{*,+}_i, t^{*,-}_i$ are supposed distinct.
\item $(\FFu,\rho)$ is a pair consisting of a rank $k$ holomorphic vector bundle $\FFu$ over  $T$ of degree $\sum_i -\floor{\theta_i^+} +\floor{\theta_i^-}$, and  a meromorphic automorphism $\rho$ of $\FFu$, with simple zeroes of type $(1,0,\ldots,0)$ at ${t^{*,+}_i}$, simple poles of type $(-1,0,\ldots,0)$ at ${t^{*,-}_i}$; the map $\rho$ is an isomorphism elsewhere. The pair  satisfies the $\widetilde\theta$ polystability condition
\item $(\FFl ,\rho)$ is a pair consisting of a rank $k$ holomorphic vector bundle $\FFl $ over $T$ of degree $n$, and a meromorphic automorphism $\rho$  of $ \FFl$, with simple zeroes of type $(1,0,\ldots,0)$ at ${t^{*,+}_i}$, simple poles of type $(-1,0,\ldots,0)$ at ${t^{*,-}_i}$; the map $\rho$ is an isomorphism elsewhere. The pair  satisfies the $\widehat\theta$ polystability condition.

\item $(C, \KKl )$ is a pair consisting of a holomorphic curve $C$ of degree $(n,k)$ in $ \PP^1 \times T^*$, and a coherent sheaf  $ \KKl $  over  $\PP^1\times T^*$ supported on $C$, with support strictly of codimension one over $\PP^1\times T^*$, and with Chern character $k[\omega_{T^*}] + n [\omega_{\PP}] + (\ell-n)[\omega_{\PP}\wedge \omega_{T^*}]$. The curve $C$ intersects $\{0\} \times T^*$ in the points $(0, t^{*,+}_i)$, and 
$\{\infty\}\times T^*$ in $(\infty, t^{*,-}_i)$; it is smooth, reduced in a neighbourhood of these points, and $\KKl $ is a line bundle over $C$ in these neighbourhoods.
\item $\EE $ is a $\theta$-semi-stable holomorphic rank $n$ vector bundle of degree $0$ over $ \PP^1 \times T$, with $c_2(\EE ) = k$. Over $0$ in $\PP^1$, the bundles $\EE $ is the sum $\oplus_i L_{t^{*,+}_i}$, and over $ \infty $, it is $\oplus_i L_{t^{*,-}_i}$
\end{itemize}
\end{theorem}

The passage from the first item to the second is the Nahm transform (Section \ref{sec:NahmTransform});  the passages from the other items to the next were shown to be bijections, with  $(\FFu,\nabla,\phi)$ to $(\FFu,\rho)$ given by a Kobayashi--Hitchin correspondence (Section \ref{sec:monopolesandpairs}), that from $(\FFu,\rho)$ to $(\FFl ,\rho)$ a Hecke transform (this Section), that from $(\FFl ,\rho)$ to $(C, \KKl )$ a spectral transform (Section \ref{sec:FourierMukai}), that from $(C, \KKl )$ to $\EE $ a Fourier--Mukai transform (Section \ref{sec:FourierMukai}). The passage from the first item to the last is given directly by a Kobayashi--Hitchin correspondence (Section \ref{sec:instantonsandholomorphicbundles}), and by the results of this last section this correspondence is equivalent to the chain of correspondences given by descending through the list.

\subsection{From $F$ to $E$: the inverse Nahm transform\ok}

There is an inverse Nahm transform, taking us from one of our singular monopoles to a an instanton on 
$\bbr\times T^3$, given by again considering the Dirac equation for the monopole. We do not examine it in detail, but  give a few brief comments. Given a singular monopole $(F,\nabla,\phi)$, one can shift the connections for the monopole by central $\uu(1)$ characters: $\nabla \rightarrow \nabla + i\mu d\theta + i\psi_1 dt_1 + i\psi_2 dt_2,\mu,\psi_1,\psi_2\in \bbr$,  and shift the Higgs field by a scalar $\phi\rightarrow \phi+ is,s\in \bbr$; hence one can define a shifted Dirac operator $\DD_{s, \mu,\psi_1,\psi_2}$. There is a periodicity which allows to think of $(\mu,\psi_1,\psi_2)$ as an element of $T^3$, and one defines a bundle over $\bbr\times T^3$ by
\[ E_{s, \mu,\psi_1,\psi_2} =\  L^2\ {\mathrm {kernel\ of}}\  \DD_{s,\mu,\psi_1,\psi_2}.\]

One  defines the connection on the bundle $E$ over  $\bbr\times T^3$ by $L^2$-projection onto the bundle of the operators $\partial/ \partial s, \partial/ \partial \mu,\partial/ \partial \psi_1,\partial/ \partial \psi_2$.

Of course, here it is the singularities of the monopole that force us to be careful; one obtains a Fredholm operator, and one can compute its index by deformations and excision as before, deforming to the $\U(1)$ Dirac monopoles in a neighbourhood of the singularities. 
Again, there is a lot of work in ensuring that the transform gives one a connection with an $L^2$ bound. One still, fortunately, has the holomorphic route; the kernel of the Dirac operator can be realized as the harmonic sections for a Dolbeault complex, and this geometry can be exploited to obtain a bijection, as we did above.

%%%%%%%%%%%%%%%%%%%%%%%%%%%%%%%%%%%%%%%%%%%%%%%%%%%%%%%
%%%%%%                                           %%%%%%
%%%%%%          SECTION                          %%%%%%
%%%%%%                                           %%%%%%
%%%%%%%%%%%%%%%%%%%%%%%%%%%%%%%%%%%%%%%%%%%%%%%%%%%%%%%
\section{Moduli} % (fold)
\label{sec:non_moduli}
\subsection{Parameter counts\ok}
We have not shown that the correspondences are homeomorphisms (they are, and we leave the proofs for elsewhere). The maps, starting from an instanton,  are fairly easily seen however to be continuous when mapping to the data on the lower line of (\ref{basic-diagram}). With this correspondence in hand, we can for example, count parameters, and obtain descriptions of the moduli. 

For example,  let us count the dimensions for the space of pairs $(C,\KKl )$ corresponding to instantons with fixed boundary conditions. For this count, one wants to count the dimension of the space of curves embedded in $\PP^1\times T^*$, with fixed intersection with $\{0\}\times T^*$ and with $\{\infty\}\times T^*$, and add to it the genus of these curves; indeed, one has a description of a generic set of the moduli as a Jacobian fibration over this family of curves. Our curves are in the linear system $kT^* + n\PP^1$; their self intersection number is $2kn$. 
The adjunction formula tells us that $N_C =  K^*_{\PP^1\times T^*}\otimes K_C$; the sections of the normal bundle corresponding to our deformations are, however, constrained to vanish at $\{0\}\times T^*$ and $\{\infty\}\times T^*$, yielding 
\[N_C (-2T^*)=   K^*_{\PP^1\times T^*}(-2T^*)\otimes K_C = K_C,\]
and identifying the infinitesimal deformations of the curve with sections of the canonical bundle. The degree of the canonical bundle $K_C$ is then $2(k-1)n$, and the genus of $C$ is then $(k-1)n +1$. The dimension of the deformations of the curve $C$  in $\PP^1\times T^*$ is then the genus. There are thus $(k-1)n +1$ parameters for the curve, and $(k-1)n +1$ parameters for what must generically be a line bundle over the curve, giving $2(k-1)n + 2$ parameters in all. 

The count, along with  additional results, was performed in \cite[Proposition 5.3]{benoitjacques3} using results from \cite{HuMa2}. 
In the case which concerns us of monopoles with $n$ singularities of weight $(1,0,\ldots,0)$ and $n$ singularities of weight $(0,0,\ldots,0,-1)$ at fixed loci, we have from \cite[Proposition 5.3]{benoitjacques3} and \cite{HuMa2}:

\begin{theorem}
\label{thm:elliptic}
The moduli space ${\MM}_s(T^*,  j_0)$
of simple pairs $(\FF ,\rho)$  with $\FF$ of degree $j_0$ and with singularities of the type given above is
smooth, of complex dimension 
$2(k-1)n + 2$. It has
a holomorphic symplectic structure, and the map 
\begin{equation}
(\FF ,\rho)\mapsto \Spec(\rho)
\end{equation}
is Lagrangian, with generic fibre a smooth compact Abelian variety, so that the moduli space has the structure of a holomorphic integrable system.
\end{theorem}

One can check that the parameter count coincides with that obtained for bundles $\EE $ over $\PP^1\times T$ of rank $n$, second Chern class $k$: here the infinitesimal deformations of the bundle are given by $H^1\bigl(\PP^1\times T, \End(\EE )\bigr)$; if the bundle is irreducible, $H^0(\PP^1\times T, \End(\EE )) =\C$, while $H^2\bigl(\PP^1\times T, \End(\EE )\bigr) = H^0\bigl(\PP^1\times T, \End(\EE )\otimes K_{\PP^1\otimes T}\bigr) =0$, so that the dimension of $H^1\bigl(\PP^1\times T, \End(\EE )\bigr)$ can be computed by Hirzeburch--Riemann--Roch, to be $2kn + 1$.\label{deformation count} The deformations for our problem are constrained, however, to be trivial over $\{0\}\times T$ and  $\{\infty\}\times T$;  these constraints impose $2n-1$ constraints. (The $-1$ is due to the fact that if the deformation of the top exterior power of $\EE $ is trivial over $\{0\}\times T$, then it is automatically trivial over  $\{\infty\}\times T$.) The number of parameters is then $2(k-1)n +2$, as above.

\subsection{Charge one\ok}

One can compute the moduli space quite easily in case of charge one. In this case, the curve $C$ is the graph of a map $\mu\colon T^*\rightarrow \PP^1$, of degree $n$. By fixing the intersections of $C$ with $\{0\}\times T^*$, $\{\infty\}\times T^*$, one is fixing the divisor $D = \sum_it_i^{*,+}-t_i^{*,-}$ of $\mu$; the fact that as an element of $T^*$, $\sum_it_i^{*,+}-t_i^{*,-} =0 $ tells us by Abel's theorem that the map $\mu$ exists. Thus
 the map $\mu$ is determined up to scale, and so up to a (free) action of $\C^*$. The other element is the sheaf $\KKl $, a line bundle supported over $C=T^*$. The map $(C,\KKl )\rightarrow \KKl \in \Pic^n(T^*)= T^*$ commutes with the action of $\C^*$, and so one has
\[\MM = \C^*\times T^*.\]  
We note that the isometry group of translations of $\RT$ acts on the moduli of instantons; moving over to the moduli of pairs  $(C,\KKl )$, the $\R\times S^1$-action becomes the $\C^*$-action; the $T$ summand acts on $\KKl $, a line bundle of degree $n$, by the natural translation action, so that the points of order $n$ in $T$ act trivially.

One can also use the holomorphic correspondence to see that some instantons do not exist. While this paper only treats the Nahm transform for instantons on $\RT$ with distinct limits at $s=\pm\infty$, the correspondence to a holomorphic object on $\PP^1\times T$ and its Fourier--Mukai transform to a pair $(C,\KKl)$ on $\PP^1\times T^*$ are well defined even when some of the $t_i^{*,+}$ coincide with some of the $t_i^{*,-}$. Let us focus on the case $n=2, k=1$. If $t_1^{*,+} = t_1^{*,-}= t^*_1$, then $t_2^{*,+} = t_2^{*,-}=t^*_2$ also, since $\sum_it_i^{*,+}-t_i^{*,-} =0 $. Suppose here that $t_1^*\neq t_2^*$ The divisor $C$ lies in the linear system $|T^*+2\PP^1|$, and so, as there are only constant maps $T^*\to \C^*$, must be of the form $( \PP^1\times \{t_1^*\})  + ( \PP^1\times \{t_2^*\})  +  (\{a\} \times T^*),$ for $a\in \C^*$.

  One can build a bundle $\EE $ with this spectral curve over $\PP^1\times T^*$ by choosing  a bundle $L$ of degree $-1$ over $T$. If $\alpha_1\colon L \rightarrow L_{t_1^{*}}, \alpha_2\colon L \rightarrow   L_{t_2^{*}}$ are non-zero maps, one takes a quotient $Q$
\[ 0 \rTo L( - \{\infty\}\times T^*) \rTo^{(z-a, \alpha_1, \alpha_2)} L\oplus L_{t_1^{*}}\oplus L_{t_2^{*}}\rTo Q\rTo 0.\]
Its restriction to $D_+=  \{0\}\times T$ is $L_{t_1^{*}}\oplus L_{t_2^{*}}$, as is its restriction to $D_- = \{\infty\}\times T$. 

The bundle $Q$ has the right structure over each $\{z\}\times T$ to give our spectral curve, but has degree $1$ in the $\PP^1$ direction. One can make one of four elementary modifications ($i= 1,2, \pm=+,-$) to bring the degree down to zero:
\[ 0\rTo \EE _{i,\pm} \rTo Q \rTo L_{t_i^{*}}|_{D_\pm}\rTo 0.\]

The bundle $\EE _{i,\pm}$ has two subsheaves, given as the images of  $L_i(-D_\pm), L_j, j\neq i$, which could destabilise it, of degree $-1,0$ respectively in the $\PP^1$ direction. Whether they do or not depends on the parabolic weights $\theta_i^\pm$. We recall that they satisfy $\theta_1^+ + \theta_2^+ -\theta_1^- -\theta_2^- = 0$. Suppose that the weights are distinct, and order them :
$\theta_1^+ -1\leq \theta_2^+ < \theta_i^- \leq\theta_j^- <\theta_1^+$. For the bundle $\EE _{2,\pm}$,  the potential destabilising bundles are $ L_{t_1^{*}}, L_{t_2^{*}}(-D_\pm)$, but these have $\theta$-degree $-\theta_1^++\theta_1^-$, and $-1-\theta_2^++\theta_2^-$, both of which are negative, and so $\EE _{2,\pm}$ is stable.   We note that the parameters are again a point $a\in \C^*$, and a line bundle $L\in \Pic^{-1}(T) = T$, and so we have the same moduli space.  

  If all the $\theta_i^+ =\theta_i^-$, the bundles we have built are semistable but not stable.
On the other hand, one can see that this construction basically exhausts possible candidates for a bundle $\EE $ with $t_1^{*+} = t_1^{*,-}= t^*_1$, $t_2^{*+} = t_2^{*,-}=t^*_2$ with $t_1^*\neq t_2^*$. Indeed, as we have seen it must have spectral curve $(\PP^1\times \{t_1^*\}) + (\PP^1\times \{t_2^{* }\}) + \{ a\} \times T^*$. Hence  $\EE|_{\{z\}\times T}\cong L_{t_1^*}\oplus L_{t_2^*}$ for a generic $z\in\PP^1$, and so if one takes the direct image of $(L_{t_i^{*}})^*\otimes \EE $ to $\PP^1$, one obtains a line bundle, which has degree $k_i$. This in turn translates into an exact sequence $L_{t_1^{*}}(-k_1D_-)\oplus L_{t_2^{*}D_-}(-k_2)\rightarrow \EE  \rightarrow {\ca S}$ defining. As the spectral curve has a component $ \{ a\} \times T^*$, taking the Fourier--Mukai transform of this sequence tells us that the sheaf ${\ca S}$ must be a line bundle supported over $ \{ a\} \times T^*$. A computation with Chern characters tells us that it is a line bundle of degree $-1$, and that $k_1+ k_2 = 1$. One thus has extensions, and these are given by the construction above. If all the $\theta_i^+ =\theta_i^-$, since  these bundles are semi-stable but not stable, one has

\begin{proposition} Let us fix the asymptotics of a connection to be the same flat connection $\nabla_\infty$ on $S^1\times T$ at both ends of the cylinder, with $\theta_i^+ = \theta_i^-$, and $t_1^{*+} = t_1^{*,-} \neq t_2^{*+} = t_2^{*,-} $. Then there are no  $\U(2)$ instantons of charge one with these asymptotics on the cylinder. 
\end{proposition}

Indeed, the fact that there are no stable bundles of charge one tells us that there are no irreducible instantons of charge one; on the other hand a reducible instanton would be a sum of $\U(1)$ instantons; as the asymptotics give topologically trivial bundles in the $T$-direction,  a sum of $\U(1)$ instantons must have charge zero.

% section moduli (end)

%%%%%%%%%%%%%%%%%%%%%%%%%%%%%%%%%%%%%%%%%%%%%%%%%%%%%%%
%%%%%%                                           %%%%%%
%%%%%%          APPENDIX                         %%%%%%
%%%%%%                                           %%%%%%
%%%%%%%%%%%%%%%%%%%%%%%%%%%%%%%%%%%%%%%%%%%%%%%%%%%%%%%
\appendix
\section{Summary of Notation} % (fold)
\label{sec:summary_of_notation}
In this appendix, we summarize the notation used in the paper.

\begin{itemize}
\item Coordinates on $S^1\times T^*$: $\theta, t^*$, with $t^*$ complex.

\item  Coordinates on  $\RT$: $s, \mu, \phi, \psi$, with $w$ a complex coordinate on $T$ combining $\phi, \psi$, and $z = \exp(-s - i \mu), \bar z = \exp(-s+i\mu)$, so that $z=0$ corresponds to $s=\infty$.

\item $t$ is the time coordinate in the heat flow.

\item $(F,\nabla,\phi)$ is a $\U(k)$ monopole, defined on $ S^1\times T^*$, with Dirac type singularities of weight    $(1,0,\ldots,0)$ at $n$ points $(  \theta^+_i, t^{*,+}_i)\in  S^1\times T^*$, and of weight    $(-1,0,\ldots,0)$ at $n$ points $(\theta^-_i, t^{*,-}_i)\in  S^1\times T^*$. The points $t^{*,+}_i, t^{*,-}_i$ are supposed distinct.

\item $(E,\nabla)$ is a rank $n$ bundle with a finite energy charge $k$ anti-self-dual $\U(n)$ connection on the flat cylinder  $\RT$. At $+\infty$,   $(E,\nabla)$ is asymptotic to a fixed flat $\U(1)^n$ connection on $S^1 \times T $ which, on each $ \{\mu\}\times T$, corresponds to the sum of line bundles $\oplus_i L_{t^{*,+}_i}$, and on the $S^1$ factors, acts on the $L_{t^{*,+}_i}$ by $\frac{\partial}{\partial u} +  i\theta^+_i$. One has, at $-\infty$, the same, but with $(t^{*,-}_i, \theta^-_i)$.

\item $(\FFl,\rho)$ is a pair consisting of a rank $k$ holomorphic vector bundle over $T$, and $\rho$ a meromorphic automorphism of $\FFl$, with simple zeroes at ${t^{*,+}_i}$, simple poles at ${t^{*,-}_i}$, and an isomorphism elsewhere. The pair satisfies a stability condition involving the $\theta^\pm_i$. The degree of $\FFl$ is $\widehat\Theta$.

\item $(\KKl,C)$ is a pair consisting of a holomorphic curve $C$ of degree $(n,k)$ in $ \PP^1 \times T^*$, and a sheaf  $K$ supported on $C$. The curve $C$ intersects $\{0\} \times T^*$ in the points $(0, t^{*,+}_i)$, and 
$\{\infty\}\times T^*$ in $(0, t^{*,-}_i)$.

\item $\EE$ is a semi-stable holomorphic rank $n$ vector bundle of degree 0 over $ \PP^1 \times T$, with $c_2(\EE ) = k$. Over $0$ in $\PP^1$, the bundles $\EE $ is the sum $\oplus_i L_{t^{*,+}_i}$, and over $\infty$, it is $\oplus_i L_{t^{*,-}_i}$

\item Basis for cohomology: Let $\omega_\PP, \omega_{S^1\times S^1}, \omega_T, \omega_{T^*}$ be  2-forms on $\PP^1,{S^1\times S^1}, T, T^*$ respectively which integrate to one, we will denote the lifts of these forms to products $\PP^1\times T$, etc. by the same letters; we will also define the two-form $\omega_\Delta$ on $T\times T^*$ as the Poincar\'e dual to the diagonal divisor.

\end{itemize}
% section summary_of_notation (end)

\subsection*{Acknowledgements.} The authors would like to thank Marcos Jardim, Tomasz Mrowka and Mark Stern for useful discussions and for their encouragements.    Concentrated bouts of joint work were accomplished during various programmes: at  the  University of Leeds, during the programme \emph{Gauge Theory and Complex Geometry} in 2011;   the Isaac Newton Institute for Mathematical Sciences,  during the programme \emph{Metric and Analytic Aspects of Moduli Spaces} in 2015; the Banff International Research Station for Mathematical Innovation and Discovery (BIRS), during their workshop 17w5149 \emph{The Analysis of Gauge-Theoretic Moduli Spaces} that took place in 2017. We thank them all for a fertile intellectual environment.
   At some point of writing, BC was on sabbatical at the Perimeter Institute for Theoretical Physics: many thanks for this quiet space and time away from teaching and administrative duties.   
BC is supported by NSERC, and  JH   by NSERC and FQRNT.  The diagrams in this paper were
created using Paul Taylor's Commutative Diagram package.

\def\cprime{$'$}\def\cprime{$'$} \def\cprime{$'$}

\end{document}